\documentclass[a4paper]{scrartcl}

\usepackage{graphicx}

\usepackage{amsmath,amssymb}
\usepackage[hypertexnames=false]{hyperref}
\usepackage{cleveref}
\usepackage{autonum}
\usepackage[vmargin=3cm]{geometry}
\usepackage[svgnames,dvipsnames]{xcolor}
\usepackage{enumitem}
\usepackage{amsthm}
\usepackage{mathtools}
\usepackage{mleftright}
\usepackage{bm}
\usepackage{upgreek}
\usepackage{xparse}
\usepackage{subcaption}
\colorlet{refkey}{pink!90!red}
\colorlet{labelkey}{JungleGreen!80!yellow}

\crefname{thm}{Theorem}{Theorems}
\newtheorem{lemma}{Lemma}[section]
\crefname{lemma}{Lemma}{Lemmas}

\crefname{cor}{Corollary}{Corollaries}

\theoremstyle{definition}
\newtheorem{scheme}{Scheme}
\crefname{scheme}{Scheme}{Schemes}
\newtheorem{remark}{Remark}[section]
\crefname{remark}{Remark}{Remarks}
\newtheorem{definition}{Definition}
\crefname{definition}{Definition}{Definitions}
\newtheorem{example}{Example}
\crefname{example}{Example}{Examples}

\crefname{figure}{Figure}{Figures}

\numberwithin{equation}{section}

\newcommand{\rd}{\mathrm{d}}
\newcommand{\bgamma}{{\bm{\upgamma}}}
\newcommand{\bnu}{\bm{\upnu}}
\newcommand{\btau}{\bm{\uptau}}

\newcommand{\grad}{\bm{\partial}}
\newcommand{\gradd}{\bm{\partial}_\rd}

\newcommand{\dtau}{\btau_\rd}

\newcommand{\bargamma}{{\bar{\bgamma}}}

\newcommand{\barg}{{\bar{g}}}
\newcommand{\bT}{\hat{\mathbf{T}}}

\DeclarePairedDelimiterXPP{\inner}[3]{}{\lparen}{\rparen}{_{#3}}{#1,#2}
\NewDocumentCommand{\innergrad}{s m m}{%
    \IfBooleanTF{#1}{%
        \inner*{\grad #2}{\grad #3}{\bgamma}
    }{%
        \inner{\grad #2}{\grad #3}{\bgamma}
    }%
}
\NewDocumentCommand{\innern}{s m m}{%
    \IfBooleanTF{#1}{%
        \inner*{#2}{#3}{\bgamma^{n-1/2}}
    }{%
        \inner{#2}{#3}{\bgamma^{n-1/2}}
    }%
}
\NewDocumentCommand{\innergradd}{s m m}{%
    \IfBooleanTF{#1}{%
        \innern*{\gradd #2}{\gradd #3}
    }{%
        \innern{\gradd #2}{\gradd #3}
    }%
}
\NewDocumentCommand{\innergrada}{s m m}{%
    \IfBooleanTF{#1}{%
        \inner*{\gradd #2}{\gradd #3}{}
    }{%
        \inner{\gradd #2}{\gradd #3}{}
    }%
}
\DeclarePairedDelimiter{\bracket}{\lbrack}{\rbrack}
\DeclarePairedDelimiter{\paren}{\lparen}{\rparen}

\newcommand{\bR}{\mathbb{R}}
\newcommand{\bN}{\mathbb{N}}

\newcommand{\bP}{\mathbf{P}}
\newcommand{\bV}{\mathbf{V}}
\newcommand{\bv}{\mathbf{v}}
\newcommand{\bw}{\mathbf{w}}

\newcommand{\dt}{\Delta t}

\allowdisplaybreaks

\title{Structure-preserving numerical methods for constrained gradient flows of planar closed curves with explicit tangential velocities}
\author{
    Tomoya Kemmochi\thanks{Graduate School of Engineering, Nagoya University. E-mail: \texttt{kemmochi@na.nuap.nagoya-u.ac.jp}}, 
    Yuto Miyatake\thanks{Cybermedia Center, Osaka University.}, 
    and 
    Koya Sakakibara\thanks{Department of Applied Mathematics, Faculty of Science, Okayama University of Science.} \thanks{RIKEN iTHEMS}
}

\begin{document}
\mleftright

\maketitle

\begin{abstract}
    In this paper, we consider numerical approximation of constrained gradient flows of planar closed curves, including the Willmore and the Helfrich flows. 
    These equations have energy dissipation and the latter has conservation properties due to the constraints.
    We will develop structure-preserving methods for these equations that preserve both the dissipation and the constraints.
    To preserve the energy structures, we introduce the discrete version of gradients according to the discrete gradient method and determine the Lagrange multipliers appropriately.
     We directly address higher order derivatives by using the Galerkin method with B-spline curves to discretize curves. 
    Moreover, we will consider stabilization of the schemes by adding tangential velocities.
    We introduce a new Lagrange multiplier to obtain both the energy structures and the stability.
    Several numerical examples are presented to verify that the proposed schemes preserve the energy structures with good distribution of control points. 
\end{abstract}

\section{Introduction}

In this paper, we consider numerical approximation of a geometric evolution equation of a planar closed curve that is described as a constrained $L^2$-gradient flow
\begin{equation}
    \bgamma_t = - \grad F_0 + \sum_{j=1}^J \lambda_j \grad F_j,
    \label{eq:intro-gf}
\end{equation}
where $\bgamma = \bgamma(u,t) \in \bR^2$ is an evolving planar closed curve with parametrization $u \in [0,1]$, 
$F_j = F_j[\bgamma]$ ($j=0,1,\dots,J$) is a functional defined for $\bgamma$, $\grad F_j = \grad F_j(\bgamma)$ is the gradient of $F_j$ in $L^2(\bgamma)$,
and $\lambda_j$ ($j=1,\dots,J$) is the Lagrange multiplier determined to fulfill the constraint that $F_j[\bgamma(t)]$ is conserved.
The number $J$ of constraints may be zero.
Typical examples that we keep in mind are the Willmore flow (or the elastic flow)
\begin{equation}
    \bgamma_t = \paren*{-2 k_{ss} - k^3 + k_0 k} \bnu
    \label{eq:intro-Willmore}
\end{equation}
with a given constant $k_0 \ge 0$, and the Helfrich flow 
\begin{equation}
    \bgamma_t =\paren*{-2 k_{ss} - k^3  -\lambda - \mu k} \bnu
    \label{eq:intro-Helfrich}
\end{equation}
with $\lambda$ and $\mu$ determined through the linear equation 
\begin{equation}
    \begin{bmatrix}
        \langle 1 \rangle & \langle k \rangle \\ 
        \langle k \rangle & \langle k^2 \rangle 
    \end{bmatrix}
    \begin{bmatrix}
        \lambda \\ \mu
    \end{bmatrix}
    = 
    \begin{bmatrix}
        -\langle k^3 \rangle \\ 
        2\langle k_s^2 \rangle - \langle k^4 \rangle 
    \end{bmatrix}.
\end{equation}
Here, 
$s$ is the arc-length parameter of $\bgamma$,
$\bnu$ is the inward unit normal vector of $\bgamma$,
$k$ is the curvature of $\bgamma$,
and $\langle f \rangle = \int_\bgamma f ds$ is the average of a scalar function $f$ over $\bgamma$.

The Willmore flow is the $L^2$-gradient flow of the elastic energy defined by 
\begin{equation}
    E[\bgamma] \coloneqq B[\bgamma] + k_0 L[\bgamma],
    \qquad B[\bgamma] \coloneqq \int_\bgamma k^2 ds, \quad L[\bgamma] \coloneqq \int_\bgamma ds,
\end{equation}
where $B$ is called the bending energy and $L$ is the length of $\bgamma$.
The critical points of the elastic energy are called Euler's elasticae, which have been studied since Bernoulli and Euler.
Hence the Willmore flow is one of the ways to find the elastica.
For the detail, we refer the reader to \cite{DKS02,BenMOS09,MPP21} and references therein.
The Helfrich flow of curves is the $L^2$-gradient flow of the bending energy $B$ under the constraints that both the length and the enclosed area of $\bgamma$ are fixed.
This flow is inspired by \cite{Hel73}, which models the shape of red blood cells as the minimizer of the bending energy of closed surfaces with fixed volume and surface area.
Furthermore, it is known that the Helfrich flow is related to shape optimization problems.
See \cite{KN06,KN0607} and also references in \cite{MPP21} for details.
We here remark that these are fourth order nonlinear evolution equations.

The aim of this paper is twofold:
\begin{itemize}
    \item To construct numerical schemes that preserve the dissipation of $F_0$ and the constraint for $F_j$ ($j \ge 1$).
    \item To stabilize the above scheme by introducing tangential velocities.
\end{itemize}
We will introduce these two topics.

\subsection{Energy structure of gradient flows}

The (constrained) gradient flow \eqref{eq:intro-gf} has energy structures, namely, dissipation of $F_0$ and conservation of $F_j$'s.
For example, when $J=1$, the Lagrange multiplier $\lambda_1$ is determined by $\lambda_1 = \innergrad{F_0}{F_1}/\innergrad{F_1}{F_1}$ to fulfill the constraint $\frac{d}{dt}F_1[\bgamma] = 0$, where $\inner{\cdot}{\cdot}{\bgamma}$ is the $L^2$-inner product over $\bgamma$.
Moreover, with this $\lambda_1$, one can see that 
\begin{equation}
    \frac{d}{dt} F_0[\bgamma] 
    = - \frac{\innergrad{F_0}{F_0} \innergrad{F_1}{F_1} - \innergrad{F_0}{F_1}^2}{\innergrad{F_1}{F_1}}
    \le 0.
    \label{eq:intro-dissipation}
\end{equation}
The same properties hold for general cases, which will be presented later.
The energy structures reflect the physical background of the equation and play important roles in the mathematical analysis.

There are several frameworks to construct numerical methods that preserve the structure of the equation such as the energy structures above.
Such methods are called structure-preserving numerical methods.
It is known that structure-preserving methods are not only physically reasonable but also advantageous for stable numerical computation with large time increments.
Therefore, a lot of techniques have been developed to construct structure-preserving methods for gradient flows such as the Allen--Cahn and the Cahn--Hilliard equations.
We refer the reader to \cite{HaiLW06,FurM11,SXY18} and references therein.
We moreover remark that some of these approaches are available for constrained cases \cite{Oku18,CheS20}.

Structure-preserving methods are also efficient for gradient flows with drastic dissipation speed.
Indeed, for such problems, the time increment should be chosen appropriately, 
and structure-preserving methods allow us to choose the time increment adaptively according to the dissipation speed.
Drastic dissipation may occur when the solution of a gradient flow passes near an unstable critical point, i.e, saddle point of the energy.
In the theory of planar curves, it is known that the elastic energy has some saddle points (cf.~\cite{Sac12,AvvKS13}).
Therefore, it is worth considering structure-preserving numerical methods for gradient flows of curves \eqref{eq:intro-gf}.

However, the mainstream of numerical methods for \eqref{eq:intro-gf} would be the parametric finite element method (cf.~\cite{DKS02,BarGN07,DecD09,BarGN10} and references therein), which is not necessarily structure-preserving.
In these papers, linear semi-implicit $P^1$-finite element schemes are proposed and their theoretical aspects are well-studied.
Although the dissipation property is discussed for the curve shortening flow ($L^2$-gradient flow for the length $L$), 
energy dissipation is not explicitly addressed for other equations such as the Willmore and the Helfrich flows.
Recently, structure-preserving parametric finite element methods for the curve diffusion flow, namely the $H^{-1}$-gradient flow for $L$, are proposed in \cite{JiaL21,BaoZ21}.
For these schemes, both dissipation of the length and conservation of the enclosed area are successfully preserved; however, it is not clear whether it is possible to extend the method to other equations.

There are several approaches other than the parametric finite element method. 
The finite difference method is widely used especially for the curve shortening flow (see e.g., \cite{Kim94,Kim97,SakM21}).
In \cite{SakM21}, structure-preserving finite difference method is proposed for the curve shortening flow with the area-preserving constraint.
Another approach is developed for unconstrained gradient flows in \cite{Kem17}, which proposes a structure-preserving method for general gradient flows \eqref{eq:intro-gf} with $J=0$. 
The method is based on the extended version of discrete gradient method (cf.~\cite{Gon96}) and the Galerkin method using the space of B-spline curves (cf.~\cite{Far99,PieT97,Sch07}).

\subsection{Tangential velocities}

Let us return to the target equation \eqref{eq:intro-gf}.
As shown in the examples \eqref{eq:intro-Willmore} and \eqref{eq:intro-Helfrich}, geometric evolution equations of curves consist of velocities in the normal direction only.
Thus there is no tangential velocity in the target equations.
However, one may add tangential velocities to the equation and replace \eqref{eq:intro-gf} by 
\begin{equation}
    \bgamma_t = - \grad F_0 + \sum_{j=1}^J \lambda_j \grad F_j + W \btau ,
    \label{eq:intro-gf-tangential}
\end{equation}
where $\btau$ is the unit tangential vector of $\bgamma$ and $W$ is an arbitrary function.
This is because the tangential velocity $W$ affects only the parameterization of curves and does not affect the evolution of shape of curves, which is proved in \cite[Proposition 2.4]{EpsG87}.
Tangential velocities neither affect the energy structure, since the gradient $\grad E(\bgamma)$ is usually perpendicular to $\btau$.
Therefore, an appropriate tangential velocity may make it simple to investigate the evolution equation (see, e.g., \cite{Dec97}).

In the context of numerical computation of evolving curves, non-trivial tangential velocities are sometimes added to enrich stability.
For example, in \cite{Kim94,Kim97}, tangential velocity that ensures uniform distribution of vertices is introduced for polygonal approximation of the curve shortening flow,
and more general one is introduced in \cite{MikS04a,MikS04b,MikS06}.
Furthermore, curvature adjusted tangential velocity, that is, tangential velocity that redistributes the vertices according to the curvature of the polygonal curve, is introduced in \cite{SevY11} and applied to image segmentation in \cite{BenKPSTY08}.

In \cite{BarGN10,BarGN11,JiaL21,BaoZ21}, the method of mass-lumping is applied to parametric finite element methods for gradient flows and it is shown that this technique redistributes the vertices of polygonal curves uniformly.
This means that tangential velocity is implicitly involved.
On the other hand, in \cite{SakM21}, tangential velocity is explicitly added to equidistribute the vertices by introducing appropriate tangential vectors.
However, it is not clear whether it is possible to extend this strategy to higher-order equations.

\subsection{Aims of this study}

The first aim of this study is to extend the approach in \cite{Kem17} to general constrained equations \eqref{eq:intro-gf}.
For the energy structures of \eqref{eq:intro-gf}, gradients of functionals play essential roles.
Indeed, from the above observation for $J=1$, both the Lagrange multiplier $\lambda$ and the dissipation speed $\frac{d}{dt} F_0[\bgamma]$ are expressed by the gradients $\grad F_0$ and $\grad F_1$.
Therefore, in order to construct structure-preserving methods for \eqref{eq:intro-gf}, it is required to define discrete version of gradients appropriately.

We will achieve this requirement using the extended version of discrete gradient method as in \cite{Kem17}.
Namely, we define the discrete gradient of a functional $F$ as the vector field $\gradd F(\bgamma,\bargamma)$ that fulfills
\begin{equation}
    F[\bgamma] - F[\bargamma] = \inner{\gradd F(\bgamma,\bargamma)}{\bgamma-\bargamma}{(\bgamma+\bargamma)/2}
\end{equation}
for all curves $\bgamma$ and $\bargamma$.
Precise definition will be given later. 

We here present the idea of our scheme for the case $J=1$ briefly.
The following idea is inspired by \cite{Oku18}.
Let $\gradd F_j(\bgamma,\bargamma)$ be the discrete gradient of $F_j$ for $j=0,1$.
Then, letting $\bgamma^n$ be an approximate solution at the $n$-th step, we temporally discretize the equation by 
\begin{equation}
    \frac{\bgamma^n - \bgamma^{n-1}}{\dt} = -\gradd F_0 + \lambda^n \gradd F_1,
    \quad
    \lambda^n = \frac{\innergradd{F_0}{F_1}}{\innergradd{F_1}{F_1}},
\end{equation}
where $\dt > 0$, $\gradd F_j = \gradd F_j(\bgamma^n, \bgamma^{n-1})$, and $\bgamma^{n-1/2} = (\bgamma^n + \bgamma^{n-1})/2$.
It is easy to see that the solution of this scheme satisfies the desired properties  
\begin{equation}
    \frac{F_0[\bgamma^n] - F_0[\bgamma^{n-1}]}{\dt} \le 0, \quad \frac{F_1[\bgamma^n] - F_1 [\bgamma^{n-1}]}{\dt} = 0.
\end{equation}
For the spatial discretization, we will employ the Galerkin method by the space of B-spline curves.
Therefore, our scheme will be described by weak formulation.

The second aim of this paper is to stabilize the structure-preserving methods by appropriate tangential velocity.
To explain the difficulty, let us consider the unconstrained gradient flow $\bgamma_t = - \grad F$, where $F$ is a given functional.
Then, our temporal discretization is 
\begin{equation}
    \frac{\bgamma^n - \bgamma^{n-1}}{\dt} = - \gradd F(\bgamma^n,\bgamma^{n-1}),
\end{equation}
which has energy dissipation 
\begin{equation}
    \frac{F[\bgamma^n] - F[\bgamma^{n-1}]}{\dt} = - \innergradd{F}{F} \le 0.
    \label{eq:intro-dissipation-disc}
\end{equation}
Now, let $\tilde{\btau}$ be, for example, the unit tangential vector of the curve $\bgamma^{n-1/2}$.
Then, the ``scheme'' with tangential velocity leads to 
\begin{equation}
    \frac{\bgamma^n - \bgamma^{n-1}}{\dt} = - \gradd F(\bgamma^n,\bgamma^{n-1}) + W^n \tilde{\btau},
\end{equation}
where $W^n$ is an arbitrary function, and this equation yields 
\begin{equation}
    \frac{F[\bgamma^n] - F[\bgamma^{n-1}]}{\dt} = - \innergradd{F}{F} + \inner{W^n \tilde{\btau}}{\gradd F}{\bgamma^{n-1/2}}.
\end{equation}
Then, the last term is a troublemaker.  
Indeed, it is not ensured that the discrete gradient $\gradd F$ is perpendicular to $\tilde{\btau}$ and thus the last term remains,
while, in the continuous case, the corresponding term vanishes owing to the orthogonality of $\grad F$ and $\btau$.

To overcome this difficulty, we will introduce a new Lagrange multiplier to ensure the dissipation property.
For the unconstrained case, we propose the scheme
\begin{equation}
    \frac{\bgamma^n - \bgamma^{n-1}}{\dt} = - \gradd F + W^n \tilde{\btau} + \lambda \grad F,
\end{equation}
and we determine $\lambda$ through the equation 
\begin{equation}
    \inner{W^n \tilde{\btau}}{\gradd F}{\bgamma^{n-1/2}} + \lambda \innergradd{F}{F} = 0.
    \label{eq:intro-lambda}
\end{equation}
Then, we can recover the dissipation property \eqref{eq:intro-dissipation-disc}.
As the function $W^n$, we use Deckelnick's tangential velocity \cite{Dec97}.
The same idea will be employed for general constrained problems.

The rest of this paper is organized as follows.
In Section~\ref{sec:preliminaries}, we collect preliminaries on the target equations \eqref{eq:intro-gf}.
Then, we propose the novel schemes in Section~\ref{sec:proposed} and present numerical examples in Section~\ref{sec:numerical}.
The movies of the numerical results are available on YouTube\footnote{\url{https://www.youtube.com/watch?v=X2gpzZT-F1M&list=PLMF3dSqWEii4L9WadrECNtPB4wgraxFJo}}.
Finally, in Section~\ref{sec:conclusion}, we give some concluding remarks.

\section{Preliminaries}
\label{sec:preliminaries}

In this section, we collect some preliminaries that will be referred in the subsequent sections.

\subsection{Structure of constrained gradient flows of planar curves}
\label{subsec:structure}

First we review the structure of our target equations, namely, constrained gradient flows of planar curves.
Let $\bgamma = \bgamma(u,t) \colon [0,1] \times [0,T] \to \bR^2$ be a time-dependent planar closed regular $C^2$-curve.
That is, assume $\bgamma$ is a periodic $C^2$-function with respect to $u$ and satisfies $|\bgamma_u(u,t)| > 0$ for all $u$ and $t$.
Here we use bold faces to describe vector-valued functions and the interval for parametrization of curves is fixed to $[0,1]$.

Let further $F$ be a functional defined for curves.
Typical examples are the area functional, the length functional, and the bending energy, which are, respectively, defined by
\begin{equation}
    A[\bgamma] \coloneqq - \frac{1}{2} \int_\bgamma \bgamma \cdot \bnu ds, \qquad
    L[\bgamma] \coloneqq \int_\bgamma ds, \qquad 
    B[\bgamma] \coloneqq \int_\bgamma k^2 ds ,
    \label{eq:functionals}
\end{equation}
where $s$ is the arc-length parameter of $\bgamma$, $\bnu$ is an inward unit normal vector of $\bgamma$, and $k$ is the (signed) curvature of $\bgamma$.
We denote the (vector-valued) gradient of the functional $F$ by $\grad F$.
That is, if $\bgamma$ is time-dependent, one has 
\begin{equation}
    \frac{d}{dt} F[\bgamma(t)] = \int_\bgamma \grad F \cdot \bgamma_t ds.
\end{equation}
Throughout this paper, we assume that $\grad F$ is parallel to $\bnu$, which is true for functionals in \eqref{eq:functionals}.
Finally, we denote the $L^2$-inner product and $L^2$-norm over $\bgamma$ by $\inner{\cdot}{\cdot}{\bgamma}$ and $\| \cdot \|_\bgamma$, respectively.

\subsubsection{One constraint case}

Let $G$ be another functional and now we consider the evolution equation 
\begin{equation}
    \bgamma_t = - \grad F + \lambda \grad G, \qquad \lambda = \frac{\innergrad{F}{G}}{\innergrad{G}{G}} , 
    \label{eq:one-constrained}
\end{equation}
which is the gradient flow of $F$ with the constraint that $G$ is conserved.
Assume that \eqref{eq:one-constrained} has a smooth solution $\bgamma$. 
Then, it is well-known that 
\begin{equation}
    \frac{d}{dt} F[\bgamma] \le 0, \qquad \frac{d}{dt} G[\bgamma] = 0.
\end{equation}
We here give a proof for later use.
Multiplying \eqref{eq:one-constrained} by $\grad G$ and integrating over $\bgamma$, we have
\begin{equation}
    \frac{d}{dt} G[\bgamma] 
    = - \innergrad{F}{G} + \frac{\innergrad{F}{G}}{\innergrad{G}{G}} \innergrad{G}{G} 
    = 0,
\end{equation}
and thus $G$ is conserved.
Notice that the Lagrange multiplier $\lambda$ is determined to enforce the constraint.
Further, multiplying \eqref{eq:one-constrained} by $\grad F$ and integrating over $\bgamma$, we have
\begin{align}
    \frac{d}{dt} F[\bgamma] 
    &= - \frac{\innergrad{F}{F}\innergrad{G}{G} - \innergrad{F}{G}^2}{\innergrad{G}{G}}, 
    \label{eq:one-dEdt}
\end{align}
and the Cauchy-Schwarz inequality yields that $F$ is dissipative.

\begin{remark}
    Multiplying \eqref{eq:one-constrained} by $\bgamma_t$ and integrating over $\bgamma$, we have
    \begin{equation}
        \inner{\bgamma_t}{\bgamma_t}{\bgamma} = - \frac{d}{dt} F[\bgamma] + \lambda \frac{d}{dt} G[\bgamma].
    \end{equation}
    Since $G$ is conserved, we obtain
    \begin{equation}
        \frac{d}{dt} F[\bgamma] = - \inner{\bgamma_t}{\bgamma_t}{\bgamma} \le 0.
    \end{equation}
    This proof is simpler and valid for multi-constrained case mentioned below.
    However, the previous proof, which states that the dissipation of $F$ is described by gradients only, plays an important role in the discrete settings, especially in the appearance of tangential velocities.
\end{remark}

\subsubsection{Two constraints case}

Let $H$ be another functional and let us next observe the evolution equation 
\begin{equation}
    \bgamma_t = - \grad F + \lambda \grad G + \mu \grad H, 
    \label{eq:two-constrained}
\end{equation}
where Lagrange multipliers $\lambda$ and $\mu$ are determined by the equations
\begin{align}
    - \innergrad{F}{G} + \lambda \innergrad{G}{G} + \mu \innergrad{H}{G} &= 0, \label{eq:two-F} \\
    - \innergrad{F}{H} + \lambda \innergrad{G}{H} + \mu \innergrad{H}{H} &= 0, \label{eq:two-G}
\end{align}
which describe the constraint that $G$ and $H$ are conserved.
For the smooth solution $\bgamma$ of \eqref{eq:two-constrained}, the functionals $F$, $G$, and $H$ satisfy
\begin{equation}
    \frac{d}{dt} F[\bgamma] \le 0, \qquad \frac{d}{dt} G[\bgamma] = 0, \qquad \frac{d}{dt} H[\bgamma] = 0.
\end{equation}
Indeed, multiplying \eqref{eq:two-constrained} by $\grad G$ and $\grad H$, we have
\begin{align}
    \frac{d}{dt} G[\bgamma] 
    &= \inner{\bgamma_t}{\grad G}{\bgamma} 
    = - \innergrad{F}{G} + \lambda \innergrad{G}{G} + \mu \innergrad{H}{G} = 0, \\
    \frac{d}{dt} H[\bgamma] 
    &= \inner{\bgamma_t}{\grad H}{\bgamma} 
    = - \innergrad{F}{H} + \lambda \innergrad{G}{H} + \mu \innergrad{H}{H} = 0, 
\end{align}
by the definition of $\lambda$ and $\mu$.
Let us observe the dissipation of $F$. 
Multiplying\eqref{eq:two-constrained} by $\grad F$, we have 
\begin{equation}
    \frac{d}{dt} F[\bgamma] = - \innergrad{F}{F} + \lambda \innergrad{G}{F} + \mu \innergrad{H}{F}.
    \label{eq:two-dEdt}
\end{equation}
Solving \eqref{eq:two-F} and \eqref{eq:two-G} as a linear system for $(\lambda,\mu)$ and substituting the result into \eqref{eq:two-dEdt}, one obtains
\begin{equation}
    \frac{d}{dt} F[\bgamma] = -
    \begin{vmatrix}
        \innergrad{F}{F} & \innergrad{F}{G} & \innergrad{F}{H} \\
        \innergrad{G}{F} & \innergrad{G}{G} & \innergrad{G}{H} \\
        \innergrad{H}{F} & \innergrad{H}{G} & \innergrad{H}{H} 
    \end{vmatrix}
    \bigg/ 
    \begin{vmatrix}
        \innergrad{G}{G} & \innergrad{G}{H} \\
        \innergrad{H}{G} & \innergrad{H}{H} 
    \end{vmatrix},
    \label{eq:two-dEdt-2}
\end{equation} 
where $|\cdot|$ denotes the determinant of a matrix. 
Since two matrices above are positive semi-definite, we can obtain the dissipation of $F$, provided that $\grad G$ and $\grad H$ are linearly independent in $L^2(\bgamma)$.

\subsubsection{General case}

The above observation can be generalized straightforwardly.
Let $F_j$ ($j=0,1,\dots,J$) be functionals and consider the evolution equation
\begin{equation}
    \bgamma_t = - \grad F_0 + \sum_{j=1}^J \lambda_j \grad F_j
    \label{eq:multi-constrained}
\end{equation}
with Lagrange multipliers $\lambda_j$ determined by
\begin{equation}
    - \innergrad{F_0}{F_i} + \sum_{j=1}^J \lambda_j \innergrad{F_j}{F_i} = 0,
    \qquad i=1,2,\dots,J.
    \label{eq:constraints}
\end{equation}
Then, the following properties hold.
\begin{lemma}\label{lem:multi-structure}
    Assume that \eqref{eq:multi-constrained} has a smooth solution $\bgamma$ and $\grad F_j$ ($j=1,\dots,J$) are linearly independent in $L^2(\bgamma)$.
    Then, we have
    \begin{align}
        \frac{d}{dt} F_0[\bgamma] &= - \frac{D(\grad F_0, \grad F_1, \dots, \grad F_J)}{D(\grad F_1, \dots, \grad F_J)} \le 0, 
        \label{eq:multi-dE0dt}
        \\ 
        \frac{d}{dt} F_i[\bgamma] &= 0, \quad i=1,2,\dots,J,
    \end{align}
    where $D(f_1,\dots,f_m) \coloneqq \det G(f_1,\dots,f_m)$ and $G(f_1,\dots,f_m) \in \bR^{m \times m}$ is the Gram matrix defined by 
    \begin{equation}
        G(f_1,\dots,f_m)_{ij} = \inner{f_i}{f_j}{\bgamma}
    \end{equation}
    for functions $f_j$ on $\bgamma$.
\end{lemma}

To see \eqref{eq:multi-dE0dt}, we show the following statement.
\begin{lemma}\label{lem:multi-lemma}
    Let $H$ be a real Hilbert space and let $f_j \in H\, (j=0,\dots,m, \, m \ge 1)$. 
    Assume that the set $(f_j)_{j=1}^m$ is linearly independent and $(f_j)_{j=0}^m$ satisfies 
    \begin{equation}
        -\inner{f_0}{f_i}{H} + \sum_{j=1}^m \lambda_j \inner{f_j}{f_i}{H} = 0, \qquad i = 1,\dots,m
        \label{eq:induction-lambdaj}
    \end{equation}
    for some $\lambda_j \in \bR$. Then, 
    \begin{equation}
        -\inner{f_0}{f_0}{H} + \sum_{j=1}^m \lambda_j \inner{f_j}{f_0}{H} = 
        - \frac{D(f_0,\dots,f_m)}{D(f_1,\dots,f_m)},
        \label{eq:induction-dE0dt}
    \end{equation}
    where $D(g_1,\dots,g_k) = \det G(g_1,\dots,g_k)$ and $G(g_1,\dots,g_k)_{ij} = \inner{g_i}{g_j}{H}$ for $g_i \in H \, (i=1,\dots,k, \, k \ge 1)$.
\end{lemma}

\begin{proof}
    Let $\Delta_{ij}$ be the $(i,j)$-cofactor of the matrix $G(f_1,\dots,f_m)$.
    Since the linear system \eqref{eq:induction-lambdaj} is solvable by the assumption, we can express $\lambda_i$ explicitly by
    \begin{equation}
        \lambda_j = \frac{1}{D(f_1,\dots,f_m)} \sum_{i=1}^m \Delta_{ij} \inner{f_0}{f_i}{H}.
    \end{equation}
    Therefore, we have
    \begin{multline}
        -\inner{f_0}{f_0}{H} + \sum_{j=1}^m \lambda_j \inner{f_j}{f_0}{H} \\
        = - \frac{1}{D(f_1,\dots,f_m)} \bracket*{ \inner{f_0}{f_0}{H}D(f_1,\dots,f_m) - \sum_{j=1}^m \inner{f_j}{f_0}{H} \sum_{i=1}^m \Delta_{ij} \inner{f_0}{f_i}{H} }.
    \end{multline}
    One can see that the factors in the bracket is nothing but $D(f_0, \dots, f_m)$, 
    and thus we complete the proof of \eqref{eq:induction-dE0dt}
\end{proof}

\begin{proof}[Proof of \cref{lem:multi-structure}]
    The conservation of $F_i$ ($i \ge 1$) is clear. Indeed, by the definition of $\lambda_j$, we have
    \begin{equation}
        \frac{d}{dt} F_i[\bgamma] 
        = \inner{\bgamma_t}{\grad F_i}{\bgamma} 
        = - \innergrad{F_0}{F_i} + \sum_{j=1}^J \lambda_j \innergrad{F_j}{F_i} = 0.
    \end{equation}
    The dissipation identity \eqref{eq:multi-dE0dt} is a direct consequence of \cref{lem:multi-lemma}.
\end{proof}

\subsection{Tangential velocity}
\label{subsec:tangential}

In general, an evolution equation of a curve is described by 
\begin{equation}
    \bgamma_t = V \bnu + W \tau,
\end{equation}
where $V$ and $W$ are normal and tangential velocities, respectively.
In \cite[Proposition 2.4]{EpsG87}, it is proved that the tangential velocity affects only the parameterization of curves and does not affect the evolution of shape of curves.
Therefore, an appropriate tangential velocity may make it simple to investigate the evolution equation.
For example, in \cite{Dec97}, the tangential velocity
\begin{equation}
    W = - \frac{\partial}{\partial u} \paren*{\frac{1}{|\bgamma_u|}}
    \label{eq:Deckelnick}
\end{equation}
is proposed to investigate the curve shortening flow. 
With this $W$, the curve shortening flow becomes a parabolic equation 
\begin{equation}
    \bgamma_t = \frac{1}{|\bgamma_u|^2} \bgamma_{uu}.
\end{equation} 
Tangential velocity is also used in the theory of numerical computation of evolving curves.
Indeed, non-trivial tangential velocities are sometimes added to control the distribution of vertices of polygonal approximations \cite{Kim94,Kim97,MikS04a,MikS04b,MikS06,SevY11,SakM21}.

We further remark that tangential velocities do not affect the geometric structure.
Let us consider the evolution equation \eqref{eq:multi-constrained} with a tangential velocity
\begin{equation}
    \bgamma_t = - \grad F_0 + \sum_{j=1}^J \lambda_j \grad F_j + W\btau
\end{equation}
and constraints \eqref{eq:constraints}.
Then, multiplying this by $\grad F_0$, we have 
\begin{equation}
    \frac{d}{dt} F_0[\bgamma] = - \innergrad{F_0}{F_0} + \sum_{j=1}^J \lambda_j \innergrad{F_j}{F_0}
\end{equation}
since $\inner{\grad F_0}{\btau}{\bgamma} = 0$.
The conservation of $F_i$ ($i \ge 1$) is also unaffected.

\section{Proposed schemes}
\label{sec:proposed}

In this section, we present our numerical schemes for constrained gradient flow \eqref{eq:multi-constrained} and \eqref{eq:constraints}.
We first consider discretization of curves.
In this study, we use B-spline curves to discretize curves, rather than polygonal approximation.
For $N,p \in \bN$, let $h = 1/N$ and 
\begin{equation}
    V_h = V_{h,p} = \operatorname{span} \{ B^p_i \}_{i=1}^N
\end{equation}
be the space of periodic B-spline functions with the know vector $\{ ih \}_{i=1}^N$, where $B^p_i$ is the $i$-th periodic B-spline basis function of degree $p$.
We also let $\bV_h = \bV_{h,p} = V_{h,p} \times V_{h,p}$ be the space of closed B-spline curves.
Therefore, a B-spline curve $\bgamma \in \bV_h$ is expressed by 
\begin{equation}
    \bgamma = \sum_{i=1}^N \bP_i B^p_i
    \label{eq:B-spline}
\end{equation}
for some $\bP_i \in \bR^2$, which is called a control point.
We discretize curves by the Galerkin approximation with the space $\bV_h$.
For precise definition, see \cite{Kem17}, and we refer the reader to \cite{Far99,PieT97,Sch07} for more details on the properties of B-spline functions.

We then consider temporal discretization.
Our method is based on the extended version of the discrete partial derivative method proposed in \cite{Kem17}, 
which can be also regarded as an extended version of the discrete gradient method.
In order to illustrate our scheme, we introduce the discrete gradient.
\begin{definition}
    Let $E$ be a functional defined over $\bV_h = \bV_{h,p}$.
    Then, discrete gradient of $E$ is a vector field $\gradd E = \gradd E(\bgamma_1,\bgamma_2) \in \bV_h$ that satisfies the following two properties:
    \begin{enumerate}[label=(\roman*)]
        \item For all $\bgamma_1,\bgamma_2 \in \bV_h$, 
        \begin{equation}
            E[\bgamma_1] - E[\bgamma_2] = \int_{(\bgamma_1+\bgamma_2)/2} \gradd E(\bgamma_1,\bgamma_2) \cdot (\bgamma_1 - \bgamma_2) ds.
            \label{eq:disc-grad}
        \end{equation}
        \item For all $\bgamma \in \bV_h$, 
        \begin{equation}
            \int_\bgamma \grad E(\bgamma) \cdot \bv ds = \int_\bgamma \gradd E(\bgamma,\bgamma) \cdot \bv ds, \qquad \forall \bv \in \bV_h.
            \label{eq:disc-grad-2}
        \end{equation}
    \end{enumerate}
\end{definition} 
We note that the definition of the discrete gradient is not unique and it requires that $\gradd E$ belongs to $\bV_h$.
The discrete gradients for the functionals mentioned in \eqref{eq:functionals} are given in \cref{app:gradients}.

We now propose a structure-preserving scheme for \eqref{eq:multi-constrained} and \eqref{eq:constraints} without tangential velocity.
The idea to define the Lagrange multipliers is inspired by \cite{Oku18}.
\begin{scheme}\label{scheme:vanila}
    For given $\bgamma^{n-1} \in \bV_h$, find $\bgamma^n \in \bV_h$ that satisfies
    \begin{equation}
        \begin{dcases}
            \innern*{\frac{\bgamma^n - \bgamma^{n-1}}{\dt_n}}{\bv}
            = - \innern{\gradd F_0}{\bv} + \sum_{j=1}^J \lambda_j \innern{\gradd F_j}{\bv}, 
            & \forall \bv \in \bV_h, \\
            - \innergradd{F_0}{F_i} + \sum_{j=1}^J \lambda_j \innergradd{F_j}{F_i} = 0, & i=1,2,\dots,J,
        \end{dcases}
        \label{eq:proposed-vanila}
    \end{equation}
    where $\dt_n > 0$, $\bgamma^{n-1/2} \coloneqq (\bgamma^n + \bgamma^{n-1})/2$, and $\gradd F_j \coloneqq \gradd F_j(\bgamma^n,\bgamma^{n-1})$.
\end{scheme}

\begin{remark}
    When there is no constraint, \cref{scheme:vanila} coincides with the scheme given in \cite{Kem17}.
\end{remark}

\cref{scheme:vanila} has the following properties, which corresponds to \cref{lem:multi-structure}. 
Substituting $\bv = \gradd F_j$, one can prove the following lemma parallel to the proof of \cref{lem:multi-structure}, which is based on \cref{lem:multi-lemma}.
Thus we omit the proof.
\begin{lemma}\label{lem:vanila}
    Assume that \cref{scheme:vanila} has a unique solution $\bgamma^n \in \bV_h$ and the vector fields $\gradd F_1, \dots, \gradd F_J$ are linearly independent. 
    Then, we have
    \begin{align}
        \frac{F_0^n - F_0^{n-1}}{\dt_n} &= - \frac{D_\rd(\gradd F_0, \gradd F_1, \dots, \gradd F_J)}{D_\rd(\gradd F_1, \dots, \gradd F_J)} \le 0, 
        \\ 
        \frac{F_i^n - F_i^{n-1}}{\dt_n} &= 0, \quad i=1,2,\dots,J,
    \end{align}
    where $F_j^n \coloneqq F_j[\bgamma^n]$ and
    \begin{equation}
        D_\rd(\gradd F_0, \gradd F_1, \dots, \gradd F_J) = \det G, \quad G_{ij} = \innergradd{F_i}{F_j}.
    \end{equation}
    The denominator $D_\rd(\gradd F_1, \dots, \gradd F_J)$ is defined by the same fashion and equal to $1$ if $J=0$.
\end{lemma}

We add a tangential velocity to \cref{scheme:vanila}.
Let $\dtau = \btau(\bgamma^{n-1/2})$ be the unit tangential vector of $\bgamma^{n-1/2}$ and $W^n$ be a given tangential velocity. 
If we add $W^n\dtau$ to \cref{scheme:vanila} naively as 
\begin{equation}
    \innern*{\frac{\bgamma^n - \bgamma^{n-1}}{\dt_n}}{\bv}
    = - \innern{\gradd F_0}{\bv} + \sum_{j=1}^J \lambda_j \innern{\gradd F_j}{\bv} + \innern{W^n\dtau}{\bv}, 
\end{equation}
then we cannot ensure the dissipation of $F_0$. 
Indeed, letting $\bv = \gradd F_0$, we have
\begin{equation}
    \frac{F_0^n - F_0^{n-1}}{\dt_n} = - \frac{D(\gradd F_0, \gradd F_1, \dots, \gradd F_J)}{D(\gradd F_1, \dots, \gradd F_J)} + \innern{W^n\dtau}{\gradd F_0}. 
\end{equation}
If we have $\innern{W^n\dtau}{\gradd F_0} = 0$ as in the continuous case, then the above system has the dissipation of $F_0$.
However, this is not the case in contrast to the continuous case. 
Indeed, the discrete gradient $\gradd F_0$, which is defined via the relation \eqref{eq:disc-grad}, is not perpendicular to $\dtau$ in general.

In order to achieve the dissipation of $F_0$ with tangential velocities, we introduce a new Lagrange multiplier $\lambda_0$ and propose the following scheme.
\begin{scheme}\label{scheme:tangential}
    For given $\bgamma^{n-1} \in \bV_h$, find $\bgamma^n \in \bV_h$ that satisfies
    
    \begin{gather}
        \begin{multlined}[b]
            \innern*{\frac{\bgamma^n - \bgamma^{n-1}}{\dt_n}}{\bv}
            = - \innern{\gradd F_0}{\bv} + \sum_{j=0}^J \lambda_j \innern{\gradd F_j}{\bv}  \\
            + \innern{W^n\dtau}{\bv}, 
            \quad \forall \bv \in \bV_h, 
            \label{eq:proposed-gf}
        \end{multlined}
        \\
        \begin{multlined}[b]
            - \innergradd{F_0}{F_0} + \sum_{j=0}^J \lambda_j \innergradd{F_j}{F_0} + \innern{W^n\dtau}{\gradd F_0} \\
                = - \frac{D_\rd(\gradd F_0, \gradd F_1, \dots, \gradd F_J)}{D_\rd(\gradd F_1, \dots, \gradd F_J)},
                \label{eq:proposed-dissipation}
        \end{multlined}
    \shortintertext{and} 
        \begin{multlined}[b]
            - \innergradd{F_0}{F_i} + \sum_{j=0}^J \lambda_j \innergradd{F_j}{F_i} + \innern{W^n\dtau}{\gradd F_i} = 0, 
            \\ i=1,2,\dots,J,
            \label{eq:proposed-constraints}
        \end{multlined}
    \end{gather}
    where $\dt_n > 0$ and the right-hand-side of \eqref{eq:proposed-dissipation} is defined as in \cref{lem:vanila}.
\end{scheme}

\begin{remark}\label{rem:multiplier}
    Letting $\dt_n \to 0$ and $h \to 0$ formally, one can see that $\lambda_0 \to 0$.
    Indeed, when $\dt_n \to 0$ and $h \to 0$, equations \eqref{eq:proposed-dissipation} and \eqref{eq:proposed-constraints} become formally 
    \begin{align}
        - \innergrad{F_0}{F_0} + \sum_{j=0}^J \lambda_j \innergrad{F_j}{F_0} 
            &= - \frac{D(\grad F_0, \grad F_1, \dots, \grad F_J)}{D(\grad F_1, \dots, \grad F_J)}, \label{eq:remark-dissipation}\\
        - \innergrad{F_0}{F_i} + \sum_{j=0}^J \lambda_j \innergrad{F_j}{F_i} &= 0, \qquad i=1,2,\dots,J. \label{eq:remark-constraints}
    \end{align}
    Now let $\tilde{\lambda}_j$ ($j=1,\dots,J$) be the solution of the linear system \eqref{eq:constraints}.
    Then, owing to \cref{lem:multi-lemma}, one can see that $ (\lambda_0,\lambda_1,\cdots,\lambda_J) = (0,\tilde{\lambda}_1,\dots,\tilde{\lambda}_J)$ satisfies \eqref{eq:remark-dissipation} and \eqref{eq:remark-constraints}.
    This means that $\lambda_0 \to 0$ formally when $\dt_n \to 0$ and $h \to 0$,
    which will be verified in numerical examples later (see \cref{ex:Willmore-lambda}). 
\end{remark}

\begin{remark}
    The role of the Lagrange multiplier $\lambda_0$ can be interpreted from the geometric viewpoint.
    To see this, let us consider the unconstrained equation $\bgamma_t = -\grad F_0$.
    Then, the naive scheme \eqref{eq:proposed-vanila} becomes 
    \begin{equation}
        \frac{\bgamma^n - \bgamma^{n-1}}{\dt_n} = - \gradd F_0
    \end{equation}
    as an equation in $\bV_h$. 
    Then, the discrete gradient $\gradd F_0$ is not perpendicular to $\dtau$ (and thus we introduce a multiplier $\lambda_0$).
    However, it is expected that they are almost perpendicular.
    Therefore, one can consider a ``modified'' tangential velocity $W^n \dtau + \lambda_0 \gradd F_0$ and determine $\lambda_0$ so that it satisfies the orthogonality in the sense of 
    \begin{equation}
        \innern{W^n \dtau + \gradd F_0}{\gradd F_0} = 0,
    \end{equation}
    which coincides with \eqref{eq:proposed-dissipation} for $J=0$.
    At this stage, $\lambda_0$ is regarded as the magnitude of modification, which is expected to be small.
\end{remark}

\cref{scheme:tangential} has the desired properties.
\begin{lemma}
    \label{lem:structure-preserving}
    Assume that \cref{scheme:tangential} has a unique solution $\bgamma^n \in \bV_h$ and the vector fields $\gradd F_1, \dots, \gradd F_J$ are linearly independent. 
    Then, we have
    \begin{align}
        \frac{F_0^n - F_0^{n-1}}{\dt_n} &= - \frac{D(\gradd F_0, \gradd F_1, \dots, \gradd F_J)}{D(\gradd F_1, \dots, \gradd F_J)} \le 0, 
        \\ 
        \frac{F_i^n - F_i^{n-1}}{\dt_n} &= 0, \quad i=1,2,\dots,J,
    \end{align}
    where $F_j^n \coloneqq F_j[\bgamma^n]$.
\end{lemma}

\begin{proof}
    By the definition of the discrete gradients and \eqref{eq:proposed-gf}, we have
    \begin{equation}
        \frac{F_i^n - F_i^{n-1}}{\dt_n} 
        = - \innergradd{F_0}{F_i} + \sum_{j=0}^J \lambda_j \innergradd{F_j}{F_i} + \innern{W^n\dtau}{\gradd F_i}
    \end{equation}
    for all $i = 0,1,\dots, J$.
    Hence we obtain the desired assertion by \eqref{eq:proposed-dissipation} and \eqref{eq:proposed-constraints}.
\end{proof}

\section{Numerical examples}
\label{sec:numerical}

In this section, we present some numerical results.
Throughout this section, we choose Deckelnick's tangential velocity \eqref{eq:Deckelnick} with magnitude $\alpha_0 > 0$ for $W$ in \cref{scheme:tangential}, namely,
\begin{equation}
    W = - \alpha_0 \frac{\partial}{\partial u} \paren*{\frac{1}{|\bgamma_u|}}.
\end{equation}
The nonlinear equation for each step is solved by the Newton method.
For the practical computation, we regard the tangential velocity $W$ as an unknown function in $V_h$.
That is, we add the equation
\begin{equation}
    \int_0^1 W v du = \alpha_0 \int_0^1 \frac{v_u}{|\bgamma_u|} du, 
    \qquad \forall v \in V_h
\end{equation}
to our scheme.
We further regard the discrete gradients and Lagrange multipliers as unknown vectors in $\bV_h$ and unknown real numbers, respectively.
Precise schemes are presented for individual examples.

Throughout this section, we use the following notations. 
Let $\bgamma^{n-1}, \bgamma^n \in \bV_h$ be the approximate solution at the $(n-1)$- and $n$-th steps, respectively.
Then, we set $\bgamma^{n-1/2} = (\bgamma^n + \bgamma^{n-1})/2$ and 
\begin{align}
    g^n &\coloneqq |\bgamma^n_u|, \qquad 
    D^n \coloneqq \det(\bgamma^n_u, \bgamma^n_{uu}), \qquad 
    \bT \coloneqq \frac{\bgamma^n_u + \bgamma^{n-1}_u}{g^n + g^{n-1}}, \\ 
    \mathbf{G} &\coloneqq \frac{(D^n)^2 + (D^{n-1})^2}{2} (g^n)^{-5} (g^{n-1})^{-5}  \sum_{l=0}^4 (g^n)^l (g^{n-1})^{4-l} \bT, \\
    H &\coloneqq \frac{(g^n)^{-5} + (g^{n-1})^{-5}}{2} \paren*{D^n + D^{n-1}},
    \qquad J \coloneqq \begin{bmatrix}
        0 & -1 \\ 1 & 0
    \end{bmatrix}. 
\end{align}

We made the movies of results of the following examples except for \cref{ex:Willmore-lambda}.
The movies are available on YouTube\footnote{\url{https://www.youtube.com/watch?v=X2gpzZT-F1M&list=PLMF3dSqWEii4L9WadrECNtPB4wgraxFJo}}.

\subsection{Unconstrained case: Willmore flow}

To see the effect of tangential velocities, we first consider an unconstrained case.
In this section, we address the Willmore flow, which is the $L^2$-gradient flow of the elastic energy defined by 
\begin{equation}
    E[\bgamma] = B[\bgamma] + k_0 L[\bgamma],
\end{equation}
where $k_0 \ge 0$ is a given constant.
That is, the Willmore flow is the evolution equation 
\begin{equation}
    \bgamma_t = - \grad E.
\end{equation}

From the formulas \eqref{eq:dL} and \eqref{eq:dB} given in Appendix, \cref{scheme:vanila} applied to the Willmore flow is as follows,
which coincides with the scheme proposed in \cite{Kem17}.
\begin{scheme}\label{scheme:Willmore-vanila}
    For given $\bgamma^{n-1} \in \bV_h$, find $\bgamma^n \in \bV_h$ and $\gradd E^n \in \bV_h$ that satisfy 
    \begin{equation}
        \innern*{\frac{\bgamma^n - \bgamma^{n-1}}{\dt_n}}{\bv}
            = - \innern{\gradd E^n}{\bv},  \qquad \forall \bv \in \bV_h
            \label{eq:scheme-Willmore}
    \end{equation}
    and 
    \begin{multline}
        \innern*{\gradd E^n}{\bw}
            = -\int_0^1 \mathbf{G}  \cdot \bw_u du 
             + \int_0^1 H J \bgamma^{n-1/2}_u \cdot \bw_{uu} du \\
             - \int_0^1 H J \bgamma^{n-1/2}_{uu} \cdot \bw_u du
             + k_0 \int_0^1 \bT \cdot \bw_u du, 
                \qquad \forall \bw \in \bV_h, 
            \label{eq:scheme-dE}
    \end{multline}
    where $\dt_n > 0$. 
\end{scheme}

\cref{scheme:tangential} applied to the Willmore flow is as follows. 
\begin{scheme}\label{scheme:Willmore-tangential}
    For given $\bgamma^{n-1} \in \bV_h$, find $\bgamma^n \in \bV_h$, $\gradd E^n \in \bV_h$, $W^n \in V_h$, and $\lambda^n \in \bR$ that satisfy four equations
    \begin{gather}
        \innern*{\frac{\bgamma^n - \bgamma^{n-1}}{\dt_n}}{\bv}
            = (-1+\lambda^n) \innern{\gradd E^n}{\bv} + \innern{W^n \dtau}{\bv},  \qquad \forall \bv \in \bV_h,\\[1ex]
        \int_0^1 W^n \varphi du = \alpha_0 \int_0^1 \frac{\varphi_u}{|\bgamma^{n-1/2}_u|} du, 
                \qquad \forall \varphi \in V_h,
        \label{eq:scheme-TV} \\[1ex]
        \lambda^n \innern{\gradd E^n}{\gradd E^n} + \innern{W^n \dtau}{\gradd E^n} = 0,
    \end{gather}
    and \eqref{eq:scheme-dE}.
\end{scheme}

\begin{example}[Willmore flow: effect of tangnetial velocities]\label{ex:Willmore}
    Let us observe numerical results of the Willmore flow with the initial curve given by
    \begin{equation}
        \bgamma(u) = r(u) \begin{bmatrix} \cos (\theta(u)) \\ \sin (\theta(u)) \end{bmatrix},
    \end{equation}
    where 
    \begin{equation}
        r(u) = 1 + 0.2 \sin (f(u)) + 0.4 \cos (f(u)), 
        \quad 
        f(u) = 2\pi u + \sin(4\pi u),
    \end{equation}
    and
    \begin{equation}
        \theta(u) = -\frac{10}{\pi} \paren*{\cos(2 \pi u) + \frac{1}{9} \cos(6 \pi u)}.
    \end{equation}

    We compare \cref{scheme:Willmore-vanila} with \cref{scheme:Willmore-tangential} in order to observe the effect of tangential velocities.
    The initial B-spline curve is given by the $L^2$-projection of the above curve onto the space $\bV_h$, 
    which is plotted in \cref{fig:Willmore-1}.
    The small circles in the figure are control points (see \eqref{eq:B-spline} in the beginning of section~\ref{sec:proposed}).
    Here, the number of control points $N$ and the degree of the B-spline function of the space $\bV_h$ are 
    \begin{equation}
        N = 25, \quad p=3.
    \end{equation}
    
    \begin{figure}
        \centering 
        \includegraphics[page=2,scale=0.8]{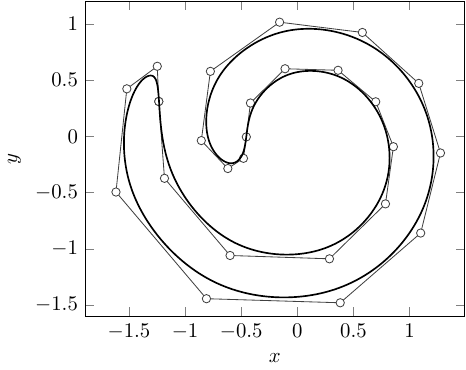}
        \caption{The initial curve for \cref{ex:Willmore}. The small circles are the control points.}
        \label{fig:Willmore-1}
    \end{figure}

    In this example, we choose time increments adaptively according to the dissipation speed.
    More precisely, we set 
    \begin{align}
        \dt_1 &= \min \left\{ \tau, \,  \| \gradd E(\bgamma^0,\bgamma^0) \|_{\bgamma^0}^{-2} \right\}, \\[1ex] 
        \dt_n &= \min \left\{ \tau, \,  \paren*{ \frac{E^{n-2} - E^{n-1}}{\dt_{n-1}} }^{-1} \right\} \quad (n \ge 2),
    \end{align}
    where $\tau>0$ is a given parameter and here we set $\tau = 10^{-4}$.
    Here, the vector $\gradd E(\bgamma^0,\bgamma^0)$ coincides with the $L^2$-projection of $\grad E(\bgamma^0)$ by \eqref{eq:disc-grad-2}. 
    The choice of $\dt_1$ is based on the relation 
    \begin{equation}
        \frac{d}{dt} E[\bgamma(t)] = - \| \grad E(\bgamma) \|_\bgamma^2,
    \end{equation}
    where $\bgamma$ is the solution of the Willmore flow.
    Our choice means $\dt_n \approx \left| \frac{dE}{dt} \right|^{-1}$ when the dissipation speed is fast.
    Finally we set the parameter of the elastic energy $k_0 = 4$, the maximum computation time $T = 0.1$, and the magnitude of the tangential velocity $\alpha_0 = 10$.

    Let us see the result by \cref{scheme:Willmore-vanila}.
    \cref{fig:Willmore-noTV-1} shows the numerical result at $t=0$ and $t \approx 0.01,0.02,0.025,0.03,0.34$.
    One can observe that the distribution of control points gets non-uniform as time passes, and there appears extremely dense parts.
    After $t \approx 0.034$, the numerical computation broke down; namely, the Newton method did not converge.

    \begin{figure}[htb]
        \centering 
        \begin{tabular}{ccc}
            \includegraphics[page=2,width=0.3\linewidth]{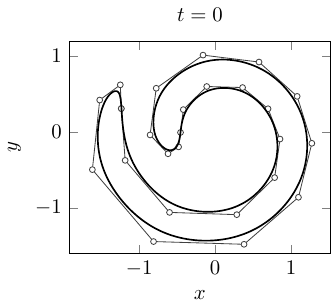}
            &
            \includegraphics[page=4,width=0.3\linewidth]{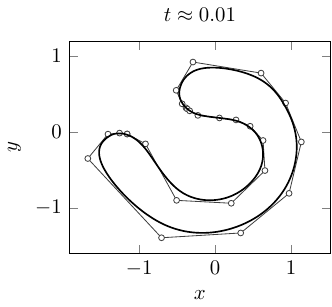}
            &
            \includegraphics[page=6,width=0.3\linewidth]{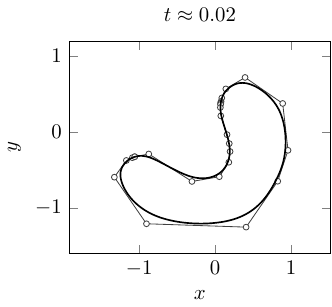}
            \\[1ex]
            \includegraphics[page=7,width=0.3\linewidth]{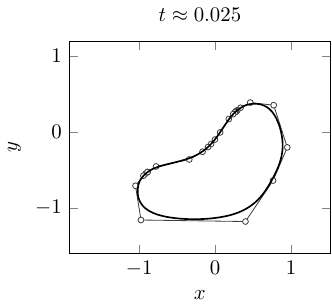}
            &
            \includegraphics[page=8,width=0.3\linewidth]{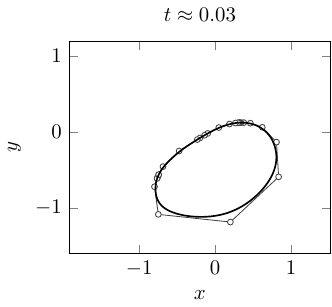}
            &
            \includegraphics[page=9,width=0.3\linewidth]{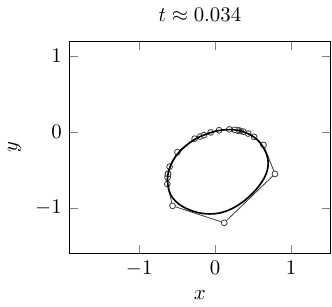}
        \end{tabular}

        \caption{Numerical result of \cref{scheme:Willmore-vanila}}
        \label{fig:Willmore-noTV-1}
    \end{figure}

    We see detailed behavior before the breakdown.
    \cref{fig:Willmore-noTV-2} shows the same numerical result at $t \approx 0.031,0.032,0.033,0.034$.
    Only control points are plotted.
    Let $\bP_j$ be the $j$-th control point for each B-spline curve.
    Then, it is observed that $\bP_{21}$ and $\bP_{22}$ got very close at $t \approx 0.031$, then overlapped at $t \approx 0.033$ and finally passed each other at $t \approx 0.034$.
    This may cause the breakdown.
    
    \begin{figure}[htb]
        \centering 
        \begin{tabular}{ccc}
            \includegraphics[page=10,scale=0.95]{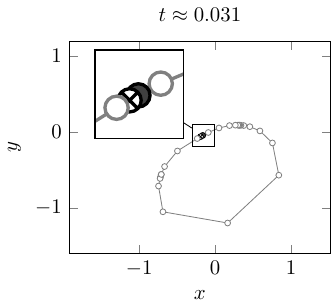}
            &
            \hspace*{1cm}
            &
            \includegraphics[page=11,scale=0.95]{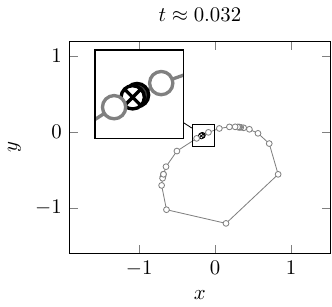}
            \\[1ex]
            \includegraphics[page=12,scale=0.95]{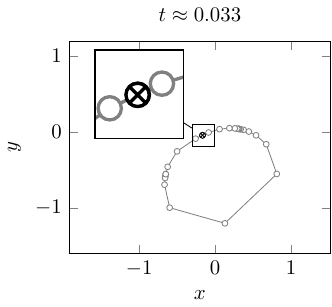}
            &

            &
            \includegraphics[page=13,scale=0.95]{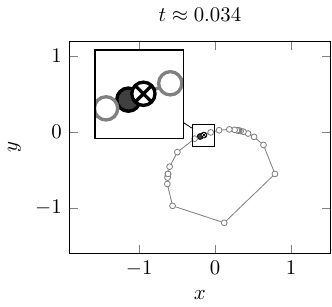}
        \end{tabular}
        
        \caption{Behavior of control points of \cref{scheme:Willmore-vanila} before breakdown. The filled circle and $\otimes$ express $\bP_{21}$ and $\bP_{22}$, respectively.}
        \label{fig:Willmore-noTV-2}
    \end{figure}

    Next we see the results by \cref{scheme:Willmore-tangential}, which is plotted in \cref{fig:Willmore-TV-1,fig:Willmore-TV-2}.
    \cref{fig:Willmore-TV-1} shows the behavior of the curve and the control points.
    One can observe that overcrowding of control points is overcome and the numerical solution is stably computed.
    The behavior of tangential velocities, which is the most important factor of \cref{scheme:Willmore-tangential}, is illustrated in \cref{fig:Willmore-TV-2}.
    One can see that the tangential velocities force control points to avoid getting close.

    \begin{figure}[htb]
        \centering 
        \begin{tabular}{ccc}
            \includegraphics[page=3,width=0.3\linewidth]{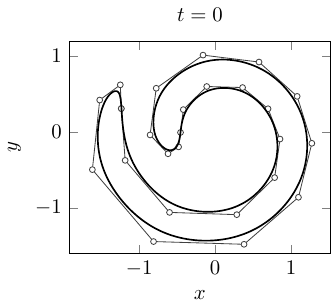}
            &
            \includegraphics[page=4,width=0.3\linewidth]{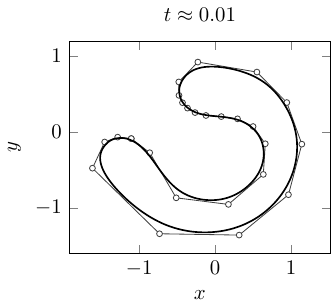}
            &
            \includegraphics[page=5,width=0.3\linewidth]{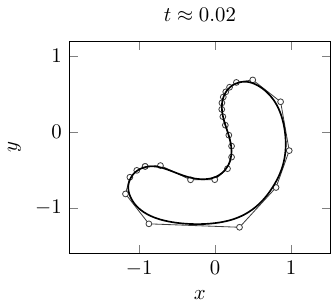}
            \\[1ex]
            \includegraphics[page=6,width=0.3\linewidth]{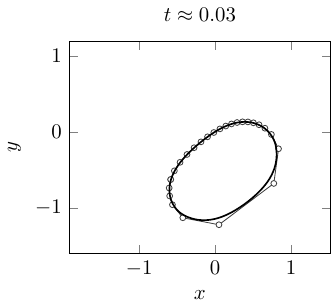}
            &
            \includegraphics[page=7,width=0.3\linewidth]{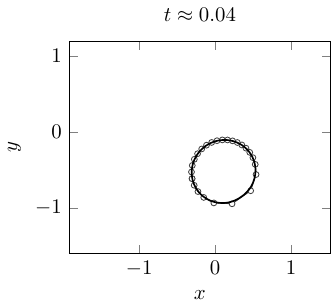}
            &
            \includegraphics[page=8,width=0.3\linewidth]{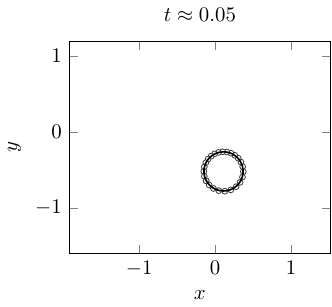}
            \\[1ex]
            \includegraphics[page=9,width=0.3\linewidth]{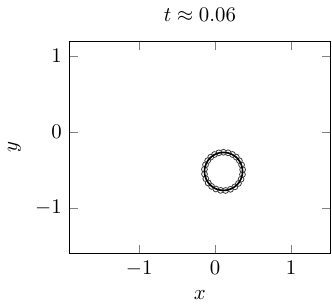}
            &
            \includegraphics[page=10,width=0.3\linewidth]{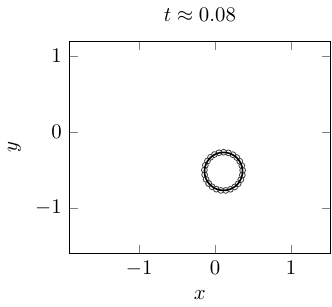}
            &
            \includegraphics[page=11,width=0.3\linewidth]{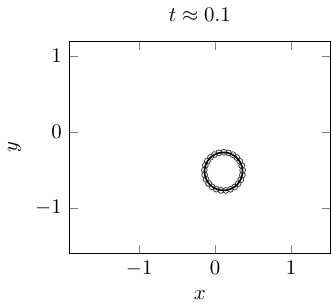}
        \end{tabular}

        \caption{Numerical result of \cref{scheme:Willmore-tangential}}
        \label{fig:Willmore-TV-1}
    \end{figure}

    \begin{figure}[htb]
        \centering 
        \begin{tabular}{ccc}
            \includegraphics[page=12,width=0.3\linewidth]{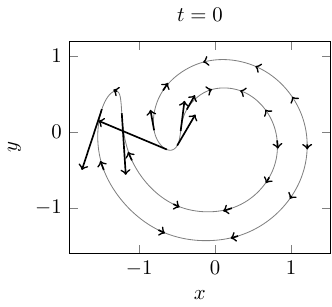}
            &
            \includegraphics[page=13,width=0.3\linewidth]{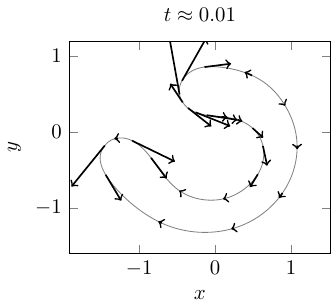}
            &
            \includegraphics[page=14,width=0.3\linewidth]{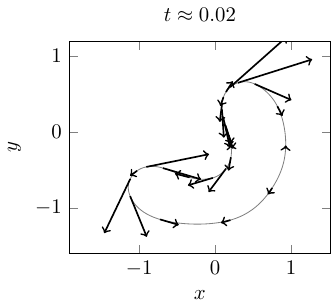}
            \\[1ex]
            \includegraphics[page=15,width=0.3\linewidth]{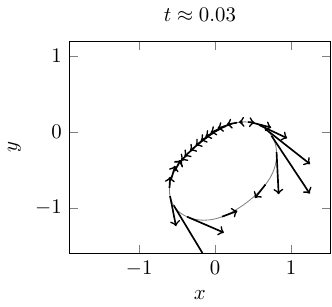}
            &
            \includegraphics[page=16,width=0.3\linewidth]{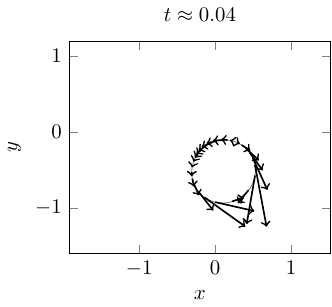}
            &
            \includegraphics[page=17,width=0.3\linewidth]{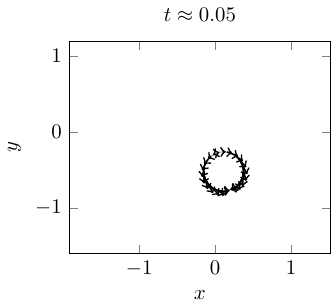}
            \\[1ex]
            \includegraphics[page=18,width=0.3\linewidth]{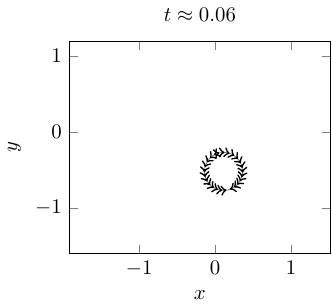}
            &
            \includegraphics[page=19,width=0.3\linewidth]{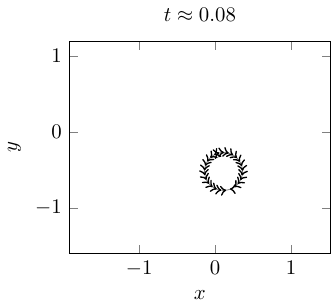}
            &
            \includegraphics[page=20,width=0.3\linewidth]{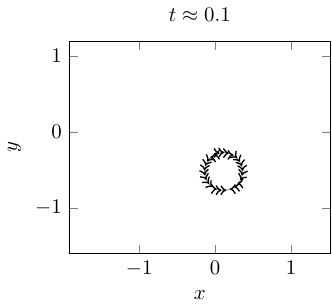}
        \end{tabular}

        \caption{Tangential velocities of the result of \cref{scheme:Willmore-tangential}. The vector fields are scaled down $1/1000$ times.}
        \label{fig:Willmore-TV-2}

    \end{figure}

    Let us see the energy dissipation.
    \cref{fig:Willmore-TV-energy} shows the evolution of the elastic energy.
    One can observe that the energy dissipation is preserved, which supports \cref{lem:structure-preserving}.
    Further, the dissipation speed is very fast when $t$ is small.
    Therefore, the time increment $\dt_n$ is smaller than $\tau = 10^{-4}$ for such steps, which is illustrated in \cref{fig:Willmore-TV-dt}.
    This example suggests that our scheme is efficient for equations with drastic dissipation.

    \begin{figure}[htb]
        \centering 
        \begin{minipage}[t]{0.45\linewidth}
            \includegraphics[page=21,scale=0.85]{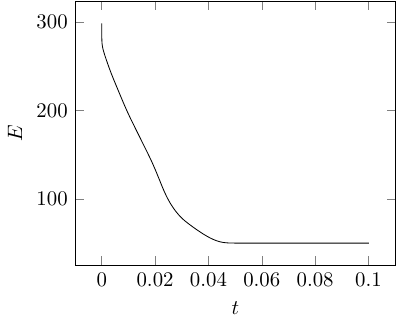}
            \caption{Evolution of the elastic energy (\cref{scheme:Willmore-tangential})}
            \label{fig:Willmore-TV-energy}
        \end{minipage}
        \hfill 
        \begin{minipage}[t]{0.45\linewidth}
            \includegraphics[page=23,scale=0.85]{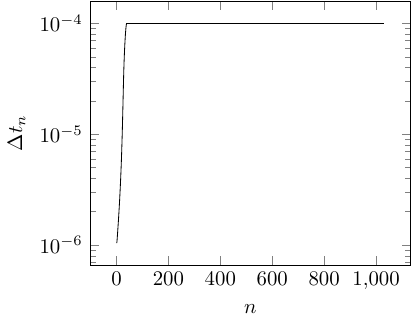}
            \caption{Time increments (\cref{scheme:Willmore-tangential})}
            \label{fig:Willmore-TV-dt}
        \end{minipage}
    \end{figure}

\end{example}

\begin{example}[Willmore flow: behavior of the Lagrange multiplier for dissipation]
    \label{ex:Willmore-lambda}
    In this example, we observe the behavior of the Lagrange multiplier $\lambda$ for energy dissipation in \cref{scheme:Willmore-tangential} when $\dt \to 0$.
    As mentioned in \cref{rem:multiplier}, it is expected that $\max_n |\lambda^n| \to 0$ when $\dt \to 0$.

    To verify this observation numerically, we computed the Willmore flow with the initial curve 
    \begin{equation}
        \bgamma(u) = 
        \begin{bmatrix}
            2 \cos(\theta(u)) \\ \sin(\theta(u))
        \end{bmatrix}
        ,\quad \theta(u) = 2\pi u - 0.4 \sin(4\pi u),
    \end{equation}
    and with parameters 
    \begin{equation}
        k_0 = 1, \quad p = 3, \quad N = 20, \quad T=4,\quad \alpha_0 = 1.
    \end{equation}
    The initial and final curves with an uniform time increment $\dt = T/400$ are plotted in \cref{fig:Willmore-lambda-curve}.
    The initial curve has dense control points around $x=\pm 2$, but the distribution of the final curve is almost uniform.

    We set $\dt = T/(100 \cdot 2^i)$ for $i=0,1,\dots,6$ and computed $\max_n |\lambda^n|$ for each $\dt$.
    The result is plotted in \cref{fig:Willmore-lambda}.
    One can see that $\max_n |\lambda^n|$ tends to decrease when $\dt \to 0$.
    However, this value saturates when $\dt \approx 10^{-3}$.
    This may be because the number of control points is relatively small; namely, rough spatial discretization may cause the saturation.

    \begin{figure}[htb]
        \centering 
        \begin{minipage}[t]{0.45\linewidth}
            \includegraphics[page=2,scale=0.9]{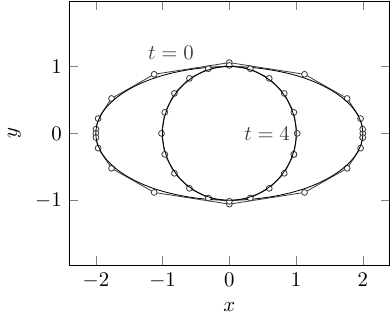}
            \caption{Solution of \cref{ex:Willmore-lambda}}
            \label{fig:Willmore-lambda-curve}
        \end{minipage}
        \hfill 
        \begin{minipage}[t]{0.45\linewidth}
            \includegraphics[page=4,scale=0.85]{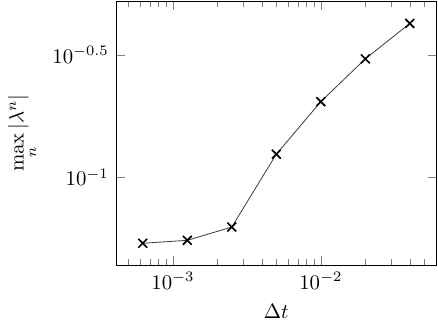}
            \caption{Relationship between $\lambda^n$ and $\dt$ (\cref{ex:Willmore-lambda})}
            \label{fig:Willmore-lambda}
        \end{minipage}
    \end{figure}

\end{example}

\subsection{One constraint case: area-preserving Willmore flow}

In this section, to see the effectiveness of \cref{scheme:tangential} with a constraint, we consider the area-preserving Willmore flow, which is the $L^2$-gradient flow of the elastic energy $E$ with a constraint that the area $A$ is preserved.
We set $k_0 = 0$ in the definition of $E$.
That is, the constrained gradient flow
\begin{equation}
    \bgamma_t = - \grad B + \lambda \grad A, \quad \lambda = \frac{\innergrad{B}{A}}{\innergrad{A}{A}}
\end{equation}
is the target of this section.

From the formulas \eqref{eq:dA} and \eqref{eq:dB}, \cref{scheme:tangential} applied to the area-preserving Willmore flow is as follows.
\begin{scheme}\label{scheme:APW-tangential}
    For given $\bgamma^{n-1} \in \bV_h$, find $\bgamma^n \in \bV_h$, $\gradd B^n \in \bV_h$, $\gradd A^n \in \bV_h$, $W^n \in V_h$, $\lambda_B^n \in \bR$, and $\lambda_A^n \in \bR$ that satisfy the six equations 
    \begin{multline}
        \innern*{\frac{\bgamma^n - \bgamma^{n-1}}{\dt_n}}{\bv}
            = (-1+\lambda^n_B) \innern{\gradd B^n}{\bv} \\ 
             + \lambda_A^n \innern{\gradd A^n}{\bv} + \innern{W^n \dtau}{\bv},  \qquad \forall \bv \in \bV_h, 
    \end{multline}
    \begin{multline}
        \innern*{\gradd B^n}{\bw}
            = - \int_0^1 \mathbf{G}  \cdot {\bw}_u du 
             + \int_0^1 H J \bgamma^{n-1/2}_u \cdot {\bw}_{uu} du 
             - \int_0^1 H J \bgamma^{n-1/2}_{uu} \cdot {\bw}_u du,\\ 
                \qquad \forall \bw \in \bV_h,
                \label{eq:scheme-dB}
    \end{multline}
    \begin{equation}
        \innern*{\gradd A^n}{\bw}
            = -\int_0^1 J \bgamma^{n-1/2} \cdot \bw du, \qquad \forall \bw \in \bV_h,
            \label{eq:scheme-dA}
    \end{equation}
    \begin{multline}
        (-1+\lambda_B^n) \innergradd{B^n}{B^n} + \lambda_A^n \innergrad{A^n}{B^n} \\
            + \innern{W^n \dtau}{\gradd B^n} = -\frac{D(\gradd B^n, \gradd A^n)}{\innergradd{A^n}{A^n}}, 
    \end{multline}
    \begin{equation}
        (-1+\lambda_B^n) \innergradd{B^n}{A^n} + \lambda_A^n \innergrad{A^n}{A^n} + \innern{W^n \dtau}{\gradd A^n} = 0,
    \end{equation}
    and \eqref{eq:scheme-TV}.
\end{scheme}

\begin{example}[Area-preserving Willmore flow]\label{ex:APW}
    Let us compute the area-preserving Willmore flow with the initial curve given by 
    \begin{equation}
        \bgamma(u) = r(u) \begin{bmatrix} \cos (\theta(u)) \\ \sin (\theta(u)) \end{bmatrix},
    \end{equation}
    where 
    \begin{equation}
        r(u) = (r_1-r_0)\frac{\cos (f(u)) +1}{2} + r_0 + \varepsilon \sin(8 \pi u)
        \quad 
        f(u) = 2\pi u + \frac{1}{2} \sin(4\pi u),
    \end{equation}
    with $r_0 = 0.5, r_1=1.5, \varepsilon=0.1$, and
    \begin{equation}
        \theta(u) = (\theta_0+\theta_1)\frac{\sin(2 \pi u)+1}{2} - \theta_1
    \end{equation}
    with $\theta_0 = 3 \pi/4, \theta_1 = \pi$.
    The $L^2$-projection onto $\bV_h$ of the initial curve is plotted in \cref{fig:APW-initial}.
    The parameters are 
    \begin{equation}
        T=4, \quad  p = 3, \quad N = 30, \quad  \alpha_0 = 1.
    \end{equation}
    We choose the time increment adaptively according to the dissipation speed as follows:
    \begin{align}
        \dt_1 &= \min \left\{ \tau, \, \frac{\innergradd{A^0}{A^0}}{D_\rd(\gradd B^0, \gradd A^0)}  \right\}, \\[1ex] 
        \dt_n &= \min \left\{ \tau, \,  \paren*{ \frac{B^{n-2} - B^{n-1}}{\dt_{n-1}} }^{-1} \right\} \quad (n \ge 2)
    \end{align}
    with  $\tau = T/1000$, where $\gradd F^0 = \gradd F(\bgamma^0,\bgamma^0) \in \bV_h$ for $F=A,B$ and 
    \begin{equation}
        D_\rd(\gradd B^0, \gradd A^0) = \begin{vmatrix}
            \innergradd{B^0}{B^0} & \innergradd{B^0}{A^0}\\
            \innergradd{B^0}{A^0} & \innergradd{A^0}{A^0}
        \end{vmatrix}.
    \end{equation}
    The definition of $\dt_1$ is based on \eqref{eq:one-dEdt}.

    \begin{figure}[htb]
        \centering\includegraphics[page=2,scale=0.8]{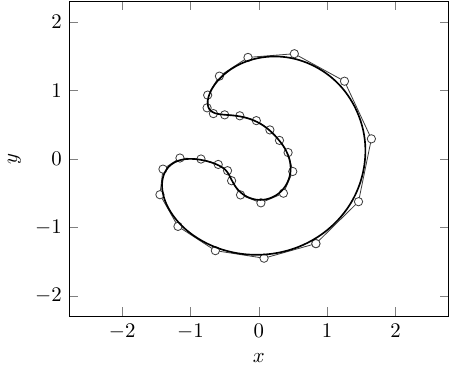}
        \caption{The initial curve for \cref{ex:APW}. The small circles are the control points.}
        \label{fig:APW-initial}
    \end{figure}

    The result is summarized in \cref{fig:APW-result,fig:APW-TV,fig:APW-functionals}.
    \cref{fig:APW-result} shows the evolution of the curve and the control points.
    One can observe that the control points are not overcrowding and finally the distribution becomes almost uniform.
    \cref{fig:APW-TV} shows the tangential velocities of the numerical solution, and it can be seen that tangential velocities provide the good distribution of control points.

    \begin{figure}[htb]
        \centering 
        \begin{tabular}{ccc}
            \includegraphics[page=3,width=0.3\linewidth]{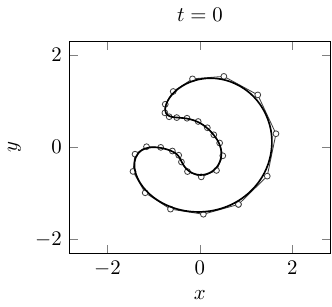}
            &
            \includegraphics[page=4,width=0.3\linewidth]{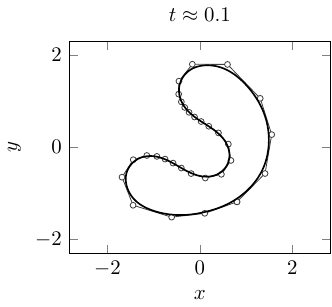}
            &
            \includegraphics[page=5,width=0.3\linewidth]{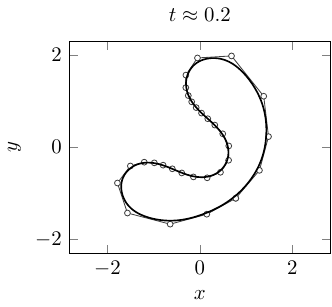}
            \\[1ex]
            \includegraphics[page=6,width=0.3\linewidth]{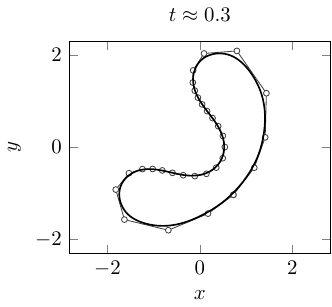}
            &
            \includegraphics[page=7,width=0.3\linewidth]{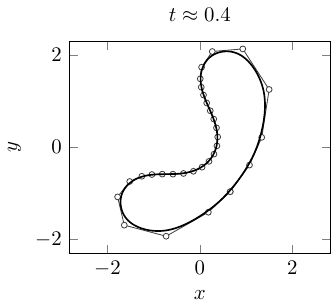}
            &
            \includegraphics[page=8,width=0.3\linewidth]{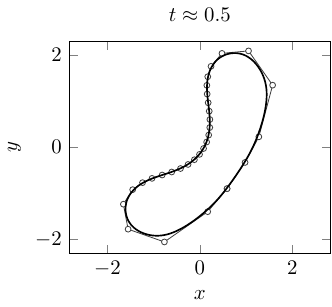}
            \\[1ex]
            \includegraphics[page=9,width=0.3\linewidth]{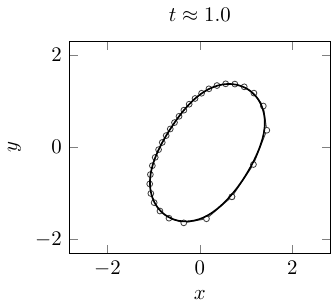}
            &
            \includegraphics[page=10,width=0.3\linewidth]{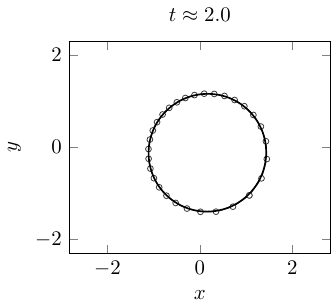}
            &
            \includegraphics[page=11,width=0.3\linewidth]{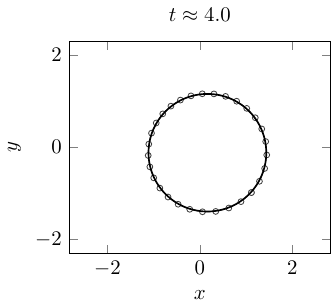}
        \end{tabular}

        \caption{Numerical results by \cref{scheme:APW-tangential}}
        \label{fig:APW-result}
    \end{figure}

    \begin{figure}[htb]
        \centering 
        \begin{tabular}{ccc}
            \includegraphics[page=12,width=0.3\linewidth]{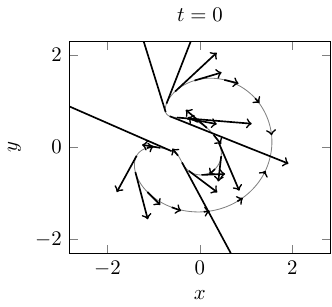}
            &
            \includegraphics[page=13,width=0.3\linewidth]{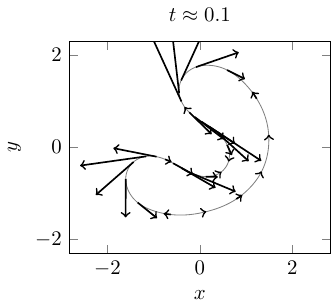}
            &
            \includegraphics[page=14,width=0.3\linewidth]{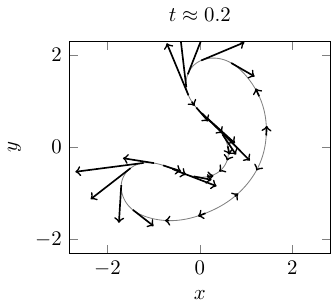}
            \\[1ex]
            \includegraphics[page=15,width=0.3\linewidth]{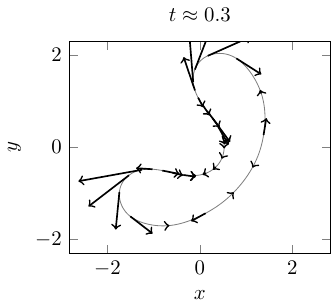}
            &
            \includegraphics[page=16,width=0.3\linewidth]{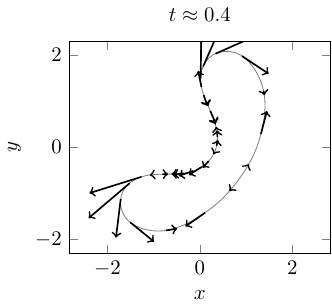}
            &
            \includegraphics[page=17,width=0.3\linewidth]{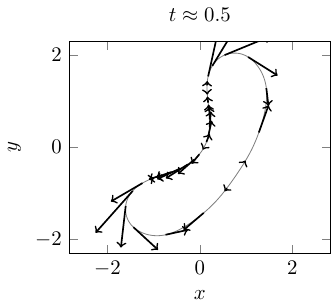}
            \\[1ex]
            \includegraphics[page=18,width=0.3\linewidth]{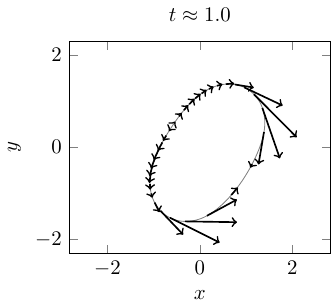}
            &
            \includegraphics[page=19,width=0.3\linewidth]{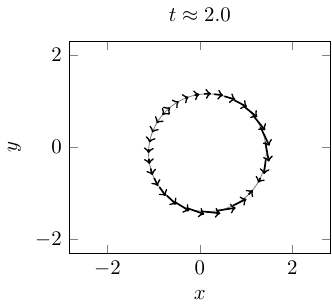}
            &
            \includegraphics[page=20,width=0.3\linewidth]{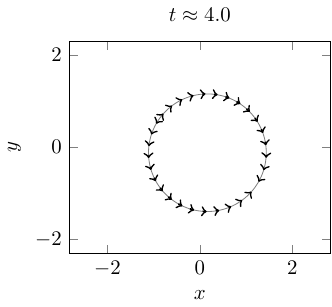}
        \end{tabular}

        \caption{Tangential velocities of the result of \cref{scheme:APW-tangential}. The vector fields are scaled down $1/200$ times.}
        \label{fig:APW-TV}

    \end{figure}

    The evolution of the bending energy $B$ and the enclosed area $A$ are presented in \cref{fig:APW-functionals}. 
    \cref{fig:APW-functionals} \subref{fig:APW-B} shows that the dissipation property is preserved.
    In \cref{fig:APW-functionals} \subref{fig:APW-A}, the relative difference of the area $(A^n-A^0)/A^0$ is plotted. 
    One can see that this value is approximately equal to machine epsilon and thus the constraint for $A$ is satisfied.
    Hence the effectiveness of \cref{scheme:APW-tangential} is observed.

    \begin{figure}[htb]
        \centering 
        \begin{subfigure}[t]{0.45\linewidth}
            \includegraphics[page=21,scale=0.85]{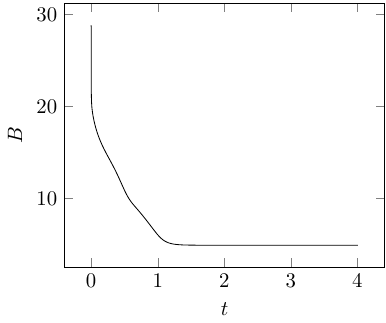}
            \caption{Evolution of $B$}
            \label{fig:APW-B}
        \end{subfigure}
        \hfill 
        \begin{subfigure}[t]{0.45\linewidth}
            \includegraphics[page=22,scale=0.85]{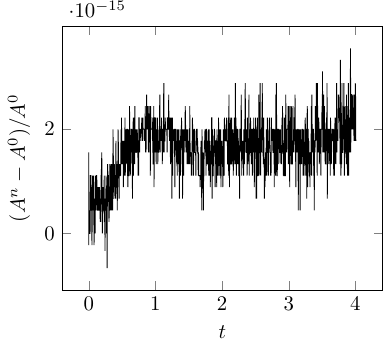}
            \caption{Evolution of $A$}
            \label{fig:APW-A}
        \end{subfigure}
        \caption{Evolution of functionals for \cref{ex:APW}}
        \label{fig:APW-functionals}
    \end{figure}

\end{example}

\subsection{Two constraints case: Helfrich flow}

In this section, we observe that \cref{scheme:tangential} is effective even for the two constraints case.
To do this, we consider the Helfrich flow, which is the gradient flow of the bending energy $B$ with a constraint that the area $A$ and the length $L$ are preserved,
which is expressed as
\begin{equation}
    \bgamma_t = - \grad B + \lambda_A \grad A + \lambda_L \grad L, 
\end{equation}
with the Lagrange multipliers determined by 
\begin{equation}
    \begin{cases}
        - \innergrad{B}{A} + \lambda_A \innergrad{A}{A} + \lambda_L \innergrad{L}{A} = 0, \\ 
        - \innergrad{B}{L} + \lambda_A \innergrad{A}{L} + \lambda_L \innergrad{L}{L} = 0. 
    \end{cases}
\end{equation}

From the formulas \eqref{eq:dA}, \eqref{eq:dL}, and \eqref{eq:dB}, \cref{scheme:tangential} applied to the Helfrich flow is as follows.
\begin{scheme}\label{scheme:Helfrich-tangential}
    For given $\bgamma^{n-1} \in \bV_h$, find $\bgamma^n \in \bV_h$, $\gradd B^n \in \bV_h$, $\gradd A^n \in \bV_h$, $\gradd L^n \in \bV_h$, $W^n \in V_h$, $\lambda_B^n \in \bR$, $\lambda_A^n \in \bR$, and $\lambda_L^n \in \bR$ that satisfy the eight equations
    \begin{multline}
        \innern*{\frac{\bgamma^n - \bgamma^{n-1}}{\dt_n}}{\bv}
            = (-1+\lambda^n_B) \innern{\gradd B^n}{\bv} \\[1ex]
             + \lambda_A^n \innern{\gradd A^n}{\bv} + \lambda_L^n \innern{\gradd L^n}{\bv} + \innern{W^n \dtau}{\bv}, \qquad \forall \bv \in \bV_h, 
    \end{multline}
    \begin{equation}
        \innern*{\gradd L^n}{\bw_L}
            = \int_0^1 \hat{\mathbf{T}} \cdot \bw_L du, \qquad \forall \bw_L \in \bV_h, 
    \end{equation}
    \begin{multline}
        (-1+\lambda_B^n) \innergradd{B^n}{B^n} + \lambda_A^n \innergrad{A^n}{B^n} \\
            + \lambda_L^n \innergrad{L^n}{B^n} + \innern{W^n \dtau}{\gradd B^n} = -\frac{D_\rd(\gradd B^n, \gradd A^n, \gradd L^n)}{D_\rd(\gradd A^n, \gradd L^n)},
    \end{multline}
    \begin{multline}
        (-1+\lambda_B^n) \innergradd{B^n}{A^n} + \lambda_A^n \innergrad{A^n}{A^n} + \lambda_L^n \innergrad{L^n}{A^n} \\ 
        + \innern{W^n \dtau}{\gradd A^n} = 0,
    \end{multline}
    \begin{multline}
        (-1+\lambda_B^n) \innergradd{B^n}{L^n} + \lambda_A^n \innergrad{A^n}{L^n} + \lambda_L^n \innergrad{L^n}{L^n} \\ 
        + \innern{W^n \dtau}{\gradd L^n} = 0,
    \end{multline}
    \eqref{eq:scheme-dB}, \eqref{eq:scheme-dA}, and \eqref{eq:scheme-TV}.
\end{scheme}

\begin{example}[Helfrich flow]\label{ex:Helfrich-1}
    We compute the Helfrich flow with the initial curve given by 
    \begin{equation}
        \bgamma(u) = r(g(u)) \begin{bmatrix} \cos (\theta(g(u))) \\ \sin (\theta(g(u))) \end{bmatrix},
        \quad 
        g(u) = u - \frac{1}{3\pi} \sin(2 \pi u),
    \end{equation}
    where $r(\cdot)$ and $\theta(\cdot)$ are the same as in \cref{ex:APW}.
    The $L^2$-projection of the initial curve is plotted in \cref{fig:Helfrich-initial}.
    The parameters are 
    \begin{equation}
        T=2, \quad  p = 3, \quad N = 30, \quad  \alpha_0 = 1.
    \end{equation}
    We choose the time increment adaptively according to the dissipation speed as follows:
    \begin{align}
        \dt_1 &= \min \left\{ \tau, \, \frac{ D_\rd (\gradd A^0, \gradd L^0)}{ D_\rd (\gradd B^0, \gradd A^0, \gradd L^0)}  \right\}, \\[1ex] 
        \dt_n &= \min \left\{ \tau, \,  \paren*{ \frac{B^{n-2} - B^{n-1}}{\dt_{n-1}} }^{-1} \right\} \quad (n \ge 2)
    \end{align}
    with  $\tau = T/2000$, where $\gradd F^0 = \gradd F(\bgamma^0,\bgamma^0) \in \bV_h$ for $F=A,L,B$ and 
    \begin{gather}
        D_\rd(\gradd A^0, \gradd L^0) = \begin{vmatrix}
            \innergrada{A^0}{A^0} & \innergrada{A^0}{L^0} \\
            \innergrada{L^0}{A^0} & \innergrada{L^0}{L^0}
        \end{vmatrix},
        \\
        D_\rd(\gradd B^0, \gradd A^0, \gradd L^0) = \begin{vmatrix}
            \innergrada{B^0}{B^0} & \innergrada{B^0}{A^0} & \innergrada{B^0}{L^0}\\
            \innergrada{A^0}{B^0} & \innergrada{A^0}{A^0} & \innergrada{A^0}{L^0}\\
            \innergrada{L^0}{B^0} & \innergrada{L^0}{A^0} & \innergrada{L^0}{L^0}
        \end{vmatrix}.
    \end{gather}
    Here, $(\cdot,\cdot)$ is the inner product over $\bgamma^{n-1/2}$.
    The definition of $\dt_1$ is based on \eqref{eq:two-dEdt-2}.

    \begin{figure}[htb]
        \centering\includegraphics[page=2,scale=0.8]{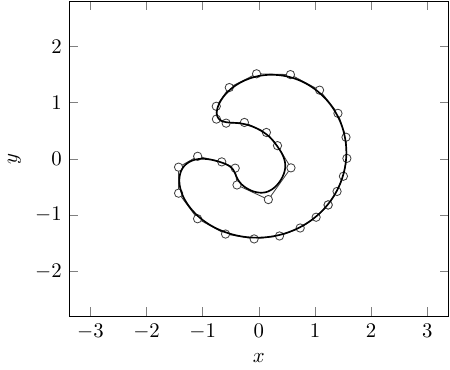}
        \caption{The initial curve for \cref{ex:Helfrich-1}. The small circles are the control points.}
        \label{fig:Helfrich-initial}
    \end{figure}

    The result is summarized in \cref{fig:Helfrich-result,fig:Helfrich-TV,fig:Helfrich-functionals}.
    \cref{fig:Helfrich-result} shows the evolution of the curve and the control points.
    The curve evolves slowly and converges to a non-trivial shape, which resembles the shape of red blood cells.
    One can observe the good behavior of the control points as in the examples above.
    \cref{fig:Helfrich-TV} shows the tangential velocities of the numerical solution, and it can be seen that tangential velocities work well.

    \begin{figure}[htb]
        \centering 
        \begin{tabular}{ccc}
            \includegraphics[page=3,width=0.3\linewidth]{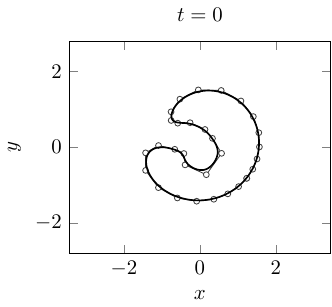}
            &
            \includegraphics[page=4,width=0.3\linewidth]{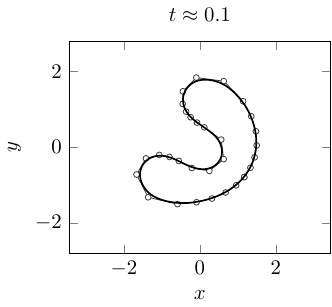}
            &
            \includegraphics[page=5,width=0.3\linewidth]{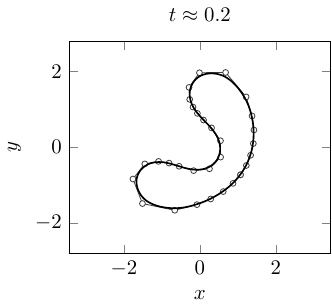}
            \\[1ex]
            \includegraphics[page=6,width=0.3\linewidth]{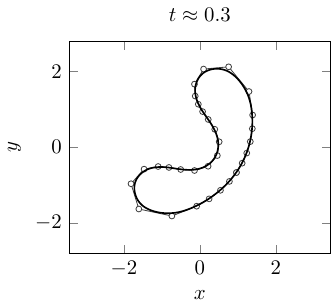}
            &
            \includegraphics[page=7,width=0.3\linewidth]{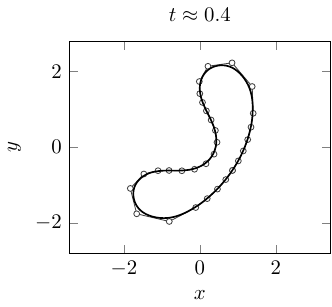}
            &
            \includegraphics[page=8,width=0.3\linewidth]{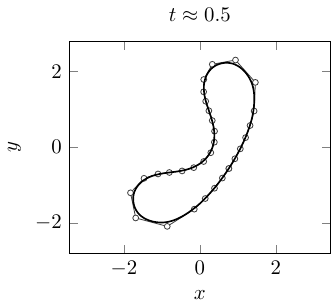}
            \\[1ex]
            \includegraphics[page=9,width=0.3\linewidth]{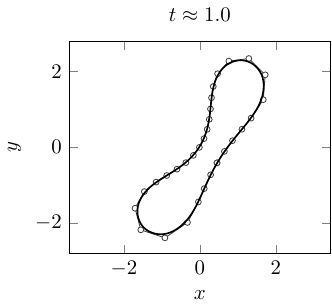}
            &
            \includegraphics[page=10,width=0.3\linewidth]{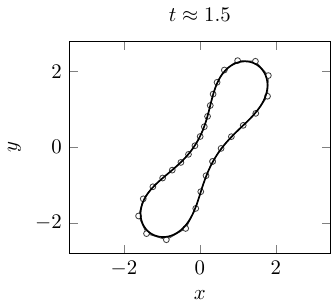}
            &
            \includegraphics[page=11,width=0.3\linewidth]{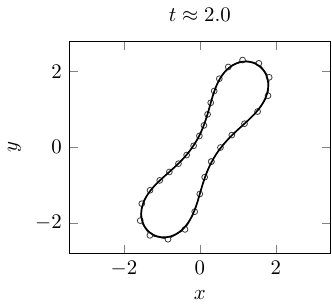}
        \end{tabular}

        \caption{Numerical results by \cref{scheme:Helfrich-tangential}}
        \label{fig:Helfrich-result}
    \end{figure}

    \begin{figure}[htb]
        \centering 
        \begin{tabular}{ccc}
            \includegraphics[page=12,width=0.3\linewidth]{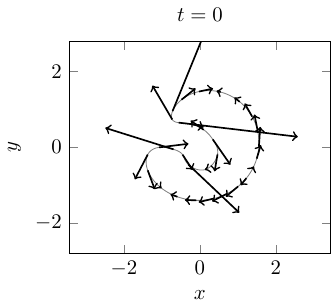}
            &
            \includegraphics[page=13,width=0.3\linewidth]{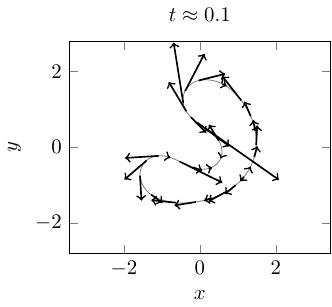}
            &
            \includegraphics[page=14,width=0.3\linewidth]{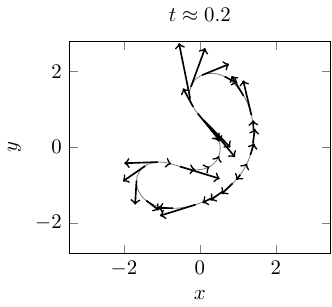}
            \\[1ex]
            \includegraphics[page=15,width=0.3\linewidth]{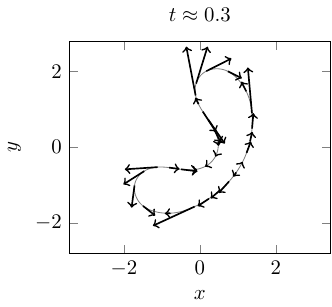}
            &
            \includegraphics[page=16,width=0.3\linewidth]{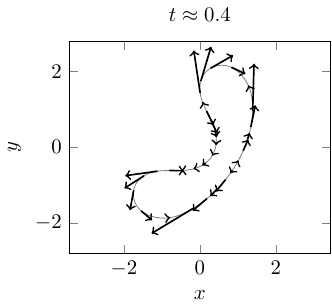}
            &
            \includegraphics[page=17,width=0.3\linewidth]{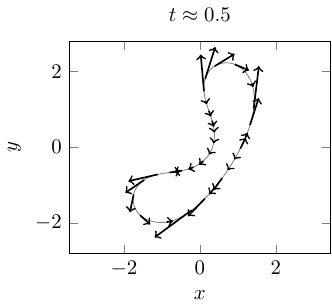}
            \\[1ex]
            \includegraphics[page=18,width=0.3\linewidth]{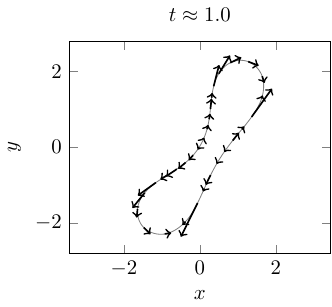}
            &
            \includegraphics[page=19,width=0.3\linewidth]{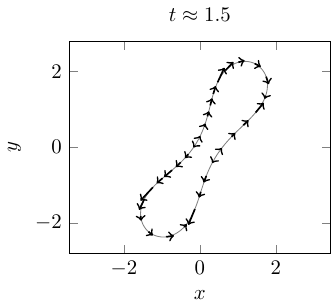}
            &
            \includegraphics[page=20,width=0.3\linewidth]{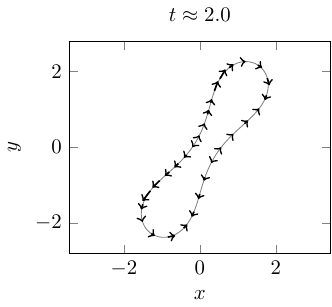}
        \end{tabular}

        \caption{Tangential velocities of the result of \cref{scheme:Helfrich-tangential}. The vector fields are scaled down $1/50$ times.}
        \label{fig:Helfrich-TV}

    \end{figure}

    The evolution of the functionals $B$, $A$, and $L$ are plotted in \cref{fig:Helfrich-functionals}.
    Figures~\ref{fig:Helfrich-functionals} \subref{fig:Helfrich-B} shows that the dissipation property is preserved.
    In \cref{fig:Helfrich-functionals} \subref{fig:Helfrich-A} and \subref{fig:Helfrich-L}, the relative differences $(F^n-F^0)/F^0$ for $F=A,L$ are plotted. 
    One can see that this value is approximately equal to machine epsilon and thus the constraints for both $A$ and $L$ are satisfied.
    Hence the effectiveness of \cref{scheme:Helfrich-tangential} can be observed even for the two-constraints case.

    \begin{figure}[htb]
        \centering 
        \begin{subfigure}[t]{0.32\linewidth}
            \includegraphics[page=21,scale=0.68]{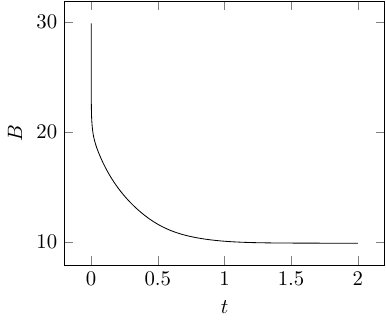}
            \caption{Evolution of $B$}
            \label{fig:Helfrich-B}
        \end{subfigure}
        \hfill 
        \begin{subfigure}[t]{0.32\linewidth}
            \includegraphics[page=22,scale=0.68]{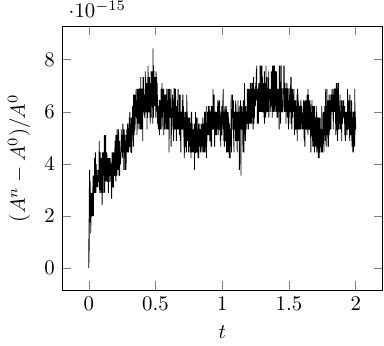}
            \caption{Evolution of $A$}
            \label{fig:Helfrich-A}
        \end{subfigure}
        \hfill 
        \begin{subfigure}[t]{0.32\linewidth}
            \includegraphics[page=23,scale=0.68]{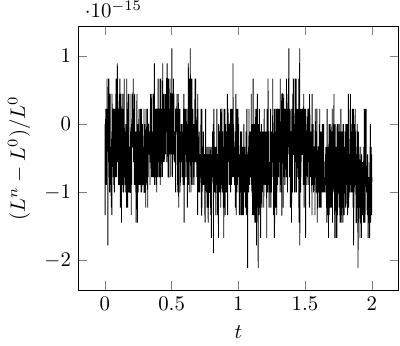}
            \caption{Evolution of $L$}
            \label{fig:Helfrich-L}
        \end{subfigure}
        \caption{Evolution of functionals for \cref{ex:Helfrich-1}}
        \label{fig:Helfrich-functionals}
    \end{figure}
\end{example}

\section{Concluding remarks}
\label{sec:conclusion}

In this paper, we constructed structure-preserving numerical methods for gradient flows of planar curves that may have one or more constraints.
Our numerical methods are based on the extended discrete gradient method introduced in \cite{Kem17}, and to preserve the constraints, we determined the Lagrange multipliers in an appropriate way.
Furthermore, we added tangential velocities to make numerical computation stable, as introduced in the literature.
The novel point is introducing a new Lagrange multiplier to attain both energy dissipation and stabilization by tangential velocities.
Numerical experiments showed that our proposed methods worked well.

However, we did not address mathematical properties of our schemes, such as well-posedness, convergence to exact solutions, and error estimates.
Furthermore, the proposed schemes are reduced to large nonlinear systems, especially for multiply constrained cases (see \cref{scheme:Helfrich-tangential}),
which yields that numerical costs of the schemes are very expensive.
Some linearization techniques proposed in \cite{MF01,SXY18} might be useful to reduce the numerical costs of our schemes.
Finally, it is not trivial whether the choice of the tangential velocity is appropriate.
Further studies are necessary for these topics.

\section*{Acknowledgments}

The first author was supported by JSPS KAKENHI Grant Numbers 19K14590 and 21H00990.
The second author was supported by JSPS KAKENHI Grant Numbers 20H00581 and 21K18301, JST PRESTO Grant Number JPMJPR2129, and ERI JURP 2022-B-06 in Earthquake Research Institute, the University of Tokyo.
The third author was supported by JSPS KAKENHI Grant Numbers 18K13455 and 22K03425.

\appendix

\section{Examples of discrete gradients}
\label{app:gradients}

In this section, we derive the discrete gradients for functionals given by \eqref{eq:functionals}.
To derive the discrete gradient of an energy $F$ according to the definition \eqref{eq:disc-grad}, we need to calculate the difference
\begin{equation}
    F[\bgamma] - F[\bargamma]
\end{equation}
for curves $\bgamma, \bargamma \in \bV_h$.
In the following, let $g = |\bgamma_u|$ be the local length of the curve $\bgamma$ and let $\barg = |\bargamma_u|$.

\subsection{Area functional}

The area functional is rewritten as 
\begin{equation}
    A[\bgamma] = -\frac{1}{2} \int_0^1 \bgamma \cdot J\bgamma_u du, 
    \qquad J = \begin{bmatrix}
        0 & -1 \\ 1 & 0
    \end{bmatrix}
    \in \bR^{2 \times 2}.
\end{equation}
Thus, we have 
\begin{align}
    A[\bgamma] - A[\bargamma] 
    &= -\frac{1}{2} \int_0^1 (\bgamma \cdot J\bgamma_u - \bargamma \cdot J\bargamma_u) du\\
    &= -\frac{1}{2} \int_0^1 \paren*{ \frac{\bgamma + \bargamma}{2} \cdot J(\bgamma_u - \bargamma_u) + (\bgamma - \bargamma) \cdot J \paren*{\frac{\bgamma_u + \bargamma_u}{2}} } du\\
    &= -\int_0^1 J \paren*{\frac{\bgamma_u + \bargamma_u}{2}} \cdot (\bgamma - \bargamma) du.
\end{align}
Here we used the skew-symmetry of $J$ and the integration by parts.
Therefore, we can define the discrete gradient $\gradd A(\bgamma,\bargamma) \in \bV_h$ implicitly by the following equation:
\begin{equation}
    \inner*{\gradd A(\bgamma,\bargamma)}{\bv}{(\bgamma+\bargamma)/2} 
    = -\int_0^1 J \paren*{\frac{\bgamma_u + \bargamma_u}{2}} \cdot \bv du,
    \quad \forall \bv \in \bV_h.
    \label{eq:dA}
\end{equation}

\subsection{Length functional}

The length functional is rewritten as 
\begin{equation}
    L[\bgamma] = \int_0^1 g du.
\end{equation}
Observe that 
\begin{equation}
    g - \barg = \frac{|\bgamma_u|^2 - |\bargamma_u|^2}{|\bgamma_u| + |\bargamma_u|}
    = \paren*{\frac{\bgamma_u + \bargamma_u}{|\bgamma_u| + |\bargamma_u|}} \cdot (\bgamma_u - \bargamma_u)
    \label{eq:dg}
\end{equation}
and let 
\begin{equation}
    \bT \coloneqq \frac{\bgamma_u + \bargamma_u}{|\bgamma_u| + |\bargamma_u|}.
\end{equation}
Then, we have 
\begin{equation}
    L[\bgamma] - L[\bargamma]
    = \int_0^1 \bT \cdot (\bgamma_u - \bargamma_u) du,
\end{equation}
which allows us to define the discrete gradient $\gradd L(\bgamma,\bargamma) \in \bV_h$ implicitly by
\begin{equation}
    \inner*{\gradd L(\bgamma,\bargamma)}{\bv}{(\bgamma+\bargamma)/2}
    = \int_0^1 \bT \cdot \bv_u du,
    \quad \forall \bv \in \bV_h.
    \label{eq:dL}
\end{equation}

\subsection{Bending energy}

The bending energy is described as
\begin{equation}
    B[\bgamma] = \int_0^1 k^2 g du = \int_0^1 \frac{\det(\bgamma_u,\bgamma_{uu})^2}{g^5} du, 
\end{equation}
where $\det(\bgamma_u,\bgamma_{uu})$ is the determinant of the matrix $\begin{pmatrix} \bgamma_u & \bgamma_{uu} \end{pmatrix} \in \bR^{2 \times 2}$.
Letting $D = \det(\bgamma_u,\bgamma_{uu})$ and  $\bar{D} = \det(\bargamma_u,\bargamma_{uu})$, we have 
\begin{align}
    B[\gamma] - B[\bargamma]
    &= \int_0^1 \paren*{ \frac{D^2}{g^5} - \frac{\bar{D}^2}{\barg^5} } du \\
    &= \int_0^1 \frac{D^2 + \bar{D}^2}{2} \paren*{g^{-5} - \barg^{-5}} du  + \int_0^1 \paren*{D^2 - \bar{D}^2} \frac{g^{-5} + \barg^{-5}}{2} du \\ 
    &\eqqcolon I_1 + I_2.
    \label{eq:dB-1}
\end{align}
We calculate the first term $I_1$. 
Using \eqref{eq:dg}, we have 
\begin{align}
    I_1 
    &= -\int_0^1 \frac{D^2 + \bar{D}^2}{2} g^{-5} \barg^{-5} \sum_{l=0}^4 g^l \barg^{4-l} (g-\barg) du\\ 
    &= -\int_0^1 \frac{D^2 + \bar{D}^2}{2} g^{-5} \barg^{-5} \sum_{l=0}^4 g^l \barg^{4-l} \bT \cdot (\bgamma_u - \bargamma_u) du.
    \label{eq:dB-2}
\end{align}
Let us address $I_2$. By the multi-linearity of the determinant, we have 
\begin{align}
    D - \bar{D} 
    &= \det \paren*{ \frac{\bgamma_u + \bargamma_u}{2}, \bgamma_{uu} - \bargamma_{uu} }
     + \det \paren*{ \bgamma_u - \bargamma_u, \frac{\bgamma_{uu} + \bargamma_{uu}}{2} } \\ 
    &= J\paren*{ \frac{\bgamma_u + \bargamma_u}{2}} \cdot \paren*{ \bgamma_{uu} - \bargamma_{uu} }
    - J\paren*{ \frac{\bgamma_{uu} + \bargamma_{uu}}{2}} \cdot \paren*{ \bgamma_u - \bargamma_u },
\end{align}
where we used the identity $\det(\mathbf{u}, \bv ) = (J \mathbf{u}) \cdot \bv$ for $\mathbf{u}, \bv \in \bR^2$.
This implies 
\begin{align}
    I_2 
    &= \int_0^1 \frac{g^{-5} + \barg^{-5}}{2} \paren*{D + \bar{D}} \paren*{D - \bar{D}}  du \\ 
    &= \int_0^1 \frac{g^{-5} + \barg^{-5}}{2} \paren*{D + \bar{D}} J\paren*{ \frac{\bgamma_u + \bargamma_u}{2}} \cdot \paren*{ \bgamma_{uu} - \bargamma_{uu} } du \\ 
    &\qquad  - \int_0^1 \frac{g^{-5} + \barg^{-5}}{2} \paren*{D + \bar{D}} J\paren*{ \frac{\bgamma_{uu} + \bargamma_{uu}}{2}} \cdot \paren*{ \bgamma_u - \bargamma_u } du.
    \label{eq:dB-3}
\end{align}
Summarizing \eqref{eq:dB-1}, \eqref{eq:dB-2}, and \eqref{eq:dB-3}, we can define the discrete gradient $\gradd B(\bgamma,\bargamma) \in \bV_h$ implicitly by
\begin{align}
    \inner*{\gradd B(\bgamma,\bargamma)}{\bv}{(\bgamma+\bargamma)/2}
    = &- \int_0^1 \frac{D^2 + \bar{D}^2}{2} g^{-5} \barg^{-5} \sum_{l=0}^4 g^l \barg^{4-l} \bT \cdot \bv_u du \\
      &+ \int_0^1 \frac{g^{-5} + \barg^{-5}}{2} \paren*{D + \bar{D}} J\paren*{ \frac{\bgamma_u + \bargamma_u}{2}} \cdot \bv_{uu} du \\
      &- \int_0^1 \frac{g^{-5} + \barg^{-5}}{2} \paren*{D + \bar{D}} J\paren*{ \frac{\bgamma_{uu} + \bargamma_{uu}}{2}} \cdot \bv_u du
    \quad \forall \bv \in \bV_h.
    \label{eq:dB}
\end{align}

\bibliographystyle{plain}
\bibliography{references}

\end{document}